\documentclass[12pt]{amsart}

\usepackage{amssymb,latexsym}

\usepackage{enumerate}

\makeatletter

\@namedef{subjclassname@2010}{

  \textup{2010} Mathematics Subject Classification}

\makeatother
\newtheorem{thm}{Theorem}[section]
\newtheorem{Con}[thm]{Conjecture}

\newtheorem{lem}[thm]{Lemma}
\newtheorem{pro}[thm]{Proposition}
\theoremstyle{definition}

\numberwithin{equation}{section}

\newcommand{\X}{\mathbf{X}}
\newcommand{\x}{\mathbf{x}}
\newcommand{\Y}{\mathbf{Y}}
\newcommand{\Z}{\mathbf{Z}}
\newcommand{\W}{\mathbf{W}}
\newcommand{\ex}{\mathbb{E}}
\newcommand{\re}{\textup{Re}}

\newcommand{\p}{\mathbb{P}}

\newcommand{\QQ}{{\mathbb Q}}
\newcommand{\RR}{{\mathbb R}}

\newcommand{\PP}{\mathbb P}
\newcommand{\EE}{\mathbb E}

\makeatletter
\renewcommand{\pmod}[1]{\allowbreak\mkern7mu({\operator@font mod}\,\,#1)}
\makeatother

\newcommand{\bal}{\[\begin{aligned}}
\newcommand{\eal}{\end{aligned}\]}

\newcommand{\be}{\begin{equation}}
\newcommand{\ee}{\end{equation}}

\newcommand{\eps}{\ensuremath{\varepsilon}}

\renewcommand{\le}{\leqslant}
\renewcommand{\leq}{\leqslant}
\renewcommand{\ge}{\geqslant}
\renewcommand{\geq}{\geqslant}

\newcommand{\fl}[1]{{\ensuremath{\left\lfloor {#1} \right\rfloor}}}

\newcommand{\order}{\asymp}      
\renewcommand{\(}{\left(}
\renewcommand{\)}{\right)}

\newcommand{\pfrac}[2]{\left(\frac{#1}{#2}\right)}  

\newcommand{\xx}{\ensuremath{\mathbf{x}}}

\newcommand{\cA}{\mathcal{A}}

\newcommand{\Var}{\mathrm{Var}}
\newcommand{\var}{\mathrm{Var}}
\begin{document}

\baselineskip=17pt

\title{Extreme biases in prime number races with many contestants}

\author{Kevin Ford}
\address{KF: Department of Mathematics, University of Illinois at Urbana-Champaign, 1409 West Green Street, Urbana, IL 61801, USA}
\email{ford@math.uiuc.edu}

\author{Adam J Harper}
\address{AJH: Mathematics Institute, Zeeman Building, University of Warwick, Coventry CV4 7AL, England}
\email{A.Harper@warwick.ac.uk}

\author{Youness Lamzouri}
\address{YL: Department of Mathematics and Statistics,
York University,
4700 Keele Street,
Toronto, ON,
M3J1P3
Canada}
\email{lamzouri@mathstat.yorku.ca}

\date{\today}


\thanks{Kevin Ford was supported by National Science Foundation grant DMS-1501982.
Adam Harper was supported, for part of this research, by a research fellowship at Jesus College, Cambridge. Youness Lamzouri is partially supported by a Discovery Grant from the Natural Sciences and Engineering Research Council of Canada. In addition, part of this work was carried out at MSRI, Berkeley during the Spring semester of 2017, supported in part by NSF grant DMS 1440140.}

\begin{abstract}
We continue to investigate the race between prime numbers in many residue classes modulo $q$, assuming the standard conjectures GRH and LI.

We show that provided $n/\log q \rightarrow \infty$ as $q \rightarrow \infty$, we can find $n$ competitor classes modulo $q$ so that the corresponding $n$-way prime number race is extremely biased. This improves on the previous range $n \geq \varphi(q)^{\epsilon}$, and (together with an existing result of Harper and Lamzouri) establishes that the transition from all $n$-way races being asymptotically unbiased, to biased races existing, occurs when $n = \log^{1+o(1)}q$.

The proofs involve finding biases in certain auxiliary races that are easier to analyse than a full $n$-way race. An important ingredient is a quantitative, moderate deviation, multi-dimensional Gaussian approximation theorem, which we prove using a Lindeberg type method.
\end{abstract}

\maketitle

\section{Introduction}
Let $q\geq 3$ and $2\leq n\leq \varphi(q)$ be integers, (where the Euler function $\varphi(q)$ denotes the number of residue classes mod $q$ that are coprime to $q$), and let $\mathcal{A}_n(q)$ be the set of ordered $n$-tuples $(a_1,a_2,\dots, a_n)$ of distinct residue classes that are coprime to $q$. In this paper we are interested in the ``Shanks--R\'enyi prime number race'', which is the following problem: if we let $\pi(x;q,a)$ denote the number of primes $p\leq x$ with $p\equiv a\bmod q$, is it true that for any $(a_1,a_2,\dots, a_n)\in \mathcal{A}_n(q)$, we will have the ordering
\begin{equation}\label{FullOrdering}
 \pi(x;q,a_1)>\pi(x;q,a_2)>\dots>\pi(x;q,a_n)
 \end{equation}
for infinitely many integers $x$? There is now an extensive body of work investigating different aspects of this question, and the reader may consult the expository papers of Granville and Martin \cite{GM},  Ford and Konyagin \cite{FK2}, and Martin and Scarfy \cite{MS} for fuller discussions.

Assuming the Generalized Riemann hypothesis GRH and the Linear Independence hypothesis LI (the assumption that the nonnegative imaginary parts of the nontrivial zeros of Dirichlet $L$-functions attached to primitive characters are linearly independent over $\mathbb{Q}$), Rubinstein and Sarnak~\cite{RuSa} proved that the answer to this question is always Yes. More strongly, they proved that for any $(a_1,\dots, a_n)\in \mathcal{A}_n(q)$, the set of real numbers $x\geq 2$ such that \eqref{FullOrdering} holds has a positive logarithmic density, which we shall denote by $\delta(q;a_1, \dots, a_n)$. Recall here that the logarithmic density of a subset $S \subseteq \mathbb{R}$ is defined as
$$ \lim_{x\to \infty}\frac{1}{\log x}\int_{t\in S\cap [2,x]}\frac{dt}{t},$$
provided the limit exists. This density can be regarded as the ``probability'' that for each $1\leq j\leq n$, the player $a_j$ is at the $j$-th position in the prime race. As we shall discuss, a probabilistic perspective turns out to be very helpful in this problem.

Next it is natural to ask whether all orderings of the $\pi(x;q,a_i)$ occur with approximately the same logarithmic density, in other words whether $\delta(q;a_1, \dots, a_n) \approx 1/n!$. For small $q$, the widely known phenomenon of {\em Chebyshev's bias} implies that $\delta(q;a_1, a_2)$ can be significantly different from 1/2, if one of the $a_i$ is a quadratic residue and the other a non-residue mod $q$. For example, Theorem 1.11 of Fiorilli and Martin~\cite{FiM} implies that $\delta(24;5, 1) \approx 0.99999$, assuming GRH and LI. On the other hand, Rubinstein and Sarnak \cite{RuSa} showed (assuming GRH and LI) that for any fixed $n$,
\begin{equation}
\lim_{q\to\infty}\max_{(a_1,\dots,a_n)\in\mathcal{A}_n(q)}\left|\delta(q;a_1,\dots,a_n)-\frac{1}{n!}\right|=0 ,
\end{equation}
in other words any biases dissolve when $q \rightarrow \infty$. Different behaviour is possible if one races ``teams'' of many residue classes combined against one another (e.g. all the quadratic residues mod $q$ against all the non-residues mod $q$), as explored in Fiorilli's paper~\cite{Fi}. Feuerverger and Martin~\cite{FeM} raised the question of having a uniform version of Rubinstein and Sarnak's statement, in which $n \to \infty$ as $q\to \infty$. And they asked whether for $n$ sufficiently large in terms of $q$ the asymptotic formula $\delta(q;a_1, \dots, a_n)\sim 1/n!$ might become false. Ford and Lamzouri (unpublished) formulated the following conjecture.
\begin{Con}[Ford and Lamzouri]\label{ExpectedRange}  Let $\varepsilon>0$ be small and $q$ be sufficiently large. 
\begin{enumerate}
\item (Uniformity for small $n$) If $2\leq n\leq (\log q)^{1-\varepsilon}$, then uniformly for all $n$-tuples $(a_1,\dots,a_n)\in\mathcal{A}_n(q)$ we have
$\delta(q;a_1, \dots, a_n)\sim 1/n!$ as $q\to \infty$.

\item (Biases for large $n$) If $ (\log q)^{1+\varepsilon} \leq n\leq \varphi(q)$, then there exist  $n$-tuples $(a_1,\dots,a_n) \in\mathcal{A}_n(q)$ and $(b_1,\dots, b_n)\in\mathcal{A}_n(q)$ for which $n! \cdot \delta(q;a_1, \dots, a_n) \to 0$ and $n! \cdot \delta(q;b_1, \dots, b_n) \to \infty$ as $q\to\infty$.

\end{enumerate} 
\end{Con}

The first part of this conjecture is now known to hold (assuming GRH and LI) in a slightly stronger form, as Harper and Lamzouri~\cite{HL} proved that, uniformly for all $2\leq n\leq \log q/(\log\log q)^4$ and all $n$-tuples $(a_1,\dots,a_n)\in\mathcal{A}_n(q)$, we have
$$\delta(q;a_1, \dots, a_n)= \frac{1}{n!}\left(1+O\left(\frac{n(\log n)^4}{\log q}\right)\right).$$
This improved on an earlier result of Lamzouri \cite{La1}, where the asymptotic $\delta(q;a_1, \dots, a_n)\sim 1/n!$ was established in the range $n = o(\sqrt{\log q})$, assuming GRH and LI.

Regarding the second part of the conjecture, Harper and Lamzouri~\cite{HL} proved (assuming GRH and LI) that for any $\varepsilon > 0$ and every $\varphi(q)^{\varepsilon} \leq n \leq \varphi(q)$, there exists an $n$-tuple $(a_1,\dots,a_n)\in\mathcal{A}_n(q)$ such  that
$$ \delta(q;a_1, \dots, a_n)<\big(1-c_{\varepsilon}\big)\frac{1}{n!},$$
where $c_{\varepsilon} > 0$ depends only on $\varepsilon$. This was the first result on $n$-way prime number races where the biases do not dissolve when $q \rightarrow \infty$, but it is clearly far from the full statement in part (2) of Ford and Lamzouri's conjecture. In particular, we note that the bias $1 - c_{\epsilon}$ is always less than 1 (whereas the conjecture asserts that biases towards both small and large values should be possible), and always close to 1 (whereas multipliers that tend to zero or to infinity with $q$ should be possible). Our goal in this paper is to revisit this issue.

\vspace{12pt}
We shall prove the following result.
\begin{thm}\label{mainbiasthm}
Assume GRH and LI. There exists a large absolute constant $C$ such that the following is true. Provided 
$n$ is sufficiently large and $n \leq \varphi(q)$, 
there exist distinct reduced residues $a_1, \cdots, a_n$ modulo $q$ such that 
$$ \delta(q;a_1, \dots, a_n) \leq \exp\left(-\frac{\min\{n, \varphi(q)^{1/50}\}}{C \log q}\right) \frac{1}{n!} , $$
and there exist distinct reduced residues $b_1, \cdots, b_n$ modulo $q$ such that 
$$ \delta(q;b_1, \dots, b_n) \geq \exp\left(\frac{\min\{n, \varphi(q)^{1/50}\}}{C \log q}\right) \frac{1}{n!} . $$
\end{thm}

Notice this fully establishes part (2) of Ford and Lamzouri's conjecture, as soon as $n/\log q \rightarrow \infty$. Furthermore, as $n$ becomes larger the relative biases become quantitatively very extreme, and for $n \leq \log q$ the bias $\exp\left( \pm \frac{n}{C \log q}\right) = 1 \pm \Theta\left(\frac{n}{\log q}\right)$ roughly matches the factor $\left(1+O\left(\frac{n(\log n)^4}{\log q}\right)\right)$ in Harper and Lamzouri's~\cite{HL} uniformity result. (For large fixed $n$, one can also think of this as clarifying the dependence on $n$ in Theorem A of Lamzouri~\cite{L3}.)

We shall actually obtain Theorem \ref{mainbiasthm} as a straightforward corollary of another ordering result. For any integers $1\leq k\leq n \leq \varphi(q)$, let us define
 $\delta_k(q;a_1, \dots, a_n)$ to be the logarithmic density of the set of real numbers $x\geq 2$ such that 
$$
\pi(x;q, a_1)>\pi(x;q, a_2) > \cdots> \pi(x;q, a_k)> \max_{k+1\leq j\leq n}\pi(x;q, a_j).
$$
If everything were uniform, we would expect that $\delta_k(q;a_1, \dots, a_n) \approx (n-k)!/n!$. Let us also define $\delta_{2k}^{\#}(q;a_1, \dots, a_n)$ to be the logarithmic density of the set of real numbers $x\geq 2$ such that
$$ \pi(x;q, a_1)>\pi(x;q, a_3) > \cdots> \pi(x;q, a_{2k-1})> \max_{2k+1\leq j\leq n}\pi(x;q, a_j) > \min_{2k+1\leq j\leq n}\pi(x;q, a_j) > $$
$$ > \pi(x;q, a_{2k}) >\cdots > \pi(x;q, a_4) > \pi(x;q, a_{2}).
$$
Again, if everything were uniform we would expect that $\delta_{2k}^{\#}(q;a_1, \dots, a_n) \approx (n-2k)!/n!$.

Harper and Lamzouri~\cite{HL} proved the uniformity result that, assuming GRH and LI,
$$ \delta_k(q;a_1, \dots, a_n)= \frac{(n-k)!}{n!}\left(1+O\left(k(\log k)^6\frac{\log n}{\log q}+\frac{1}{n \log^{1/10}q} \right)\right) $$
whenever $k(\log k)^{10} \leq (\log q)/\log n$. They also proved a non-uniformity result: for any fixed\footnote{We remark that the case $k=1$ is not excluded from the non-uniformity result for technical reasons, but because it is genuinely different. Harper and Lamzouri~\cite{HL} showed that $\delta_1(q;a_1, \dots, a_n) \sim 1/n$ provided $n = o(\varphi(q)^{1/32})$, when $q \rightarrow \infty$. Notice that Theorem \ref{kExtremeBias}, below, deals with $\delta_{2k}(q;a_1, \dots, a_n)$ and $\delta_{2k}^{\#}(q;a_1, \dots, a_n)$, so the case of $\delta_1(q;a_1, \dots, a_n)$ is again excluded.} $k \geq 2$, 
fixed $\eps>0$, 
and any $\varphi(q)^{\varepsilon} \leq n<\varphi(q)^{1/41}$, there exists an $n$-tuple $(a_1,\dots,a_n)\in\mathcal{A}_n(q)$ such that
$$ \delta_k(q;a_1, \dots, a_n)<\big(1-c_{\varepsilon}\big)\frac{(n-k)!}{n!} . $$

Here we establish a significantly improved non-uniformity result.

\begin{thm}\label{kExtremeBias} Assume GRH and LI. There exists a large absolute constant $A$ such that the following is true. Suppose $q$ is large, and let $1 \leq k \leq n/A \leq \varphi(q)^{1/50}$. Then there exist distinct reduced residues $a_1, \cdots, a_n$ modulo $q$ such that 
\begin{equation}\label{SmallBias}
\delta_{2k}(q;a_1, \dots, a_n) \leq \exp\left(-\frac{k\log (n/k)}{2\log q}\right) \frac{(n-2k)!}{n!}, 
\end{equation}
and 
\begin{equation}\label{LargeBias}
\delta_{2k}^{\#}(q;a_1, \dots, a_n) \geq \exp\left(\frac{k\log (n/k)}{2\log q}\right) \frac{(n-2k)!}{n!}.
\end{equation}
\end{thm}

\begin{proof}[Proof of Theorem \ref{mainbiasthm}, assuming Theorem \ref{kExtremeBias}]
Suppose first that $n \leq \varphi(q)^{1/50}$. Take $k= \lceil n/2A \rceil$, (which indeed satisfies $1 \leq k \leq n/A$ since we assume in Theorem \ref{mainbiasthm} that $n$ is large), and note that by Theorem \ref{kExtremeBias} we have
$$ \delta_{2k}(q;a_1, \dots, a_n) \leq \exp\left(-\frac{n \log(2A)}{4A \log q}\right) \frac{(n-2k)!}{n!} . $$
On the other hand, since the logarithmic density of the set of real numbers $x\geq 2$ for which $\pi(x;q,a)=\pi(x;q,b)$ is $0$ (which follows from equation \eqref{DensityMeasure} below), we get
$$ \delta_{2k}(q; a_1, \dots, a_n)= \sum_{\sigma\in S_{n-2k}}\delta(q; a_1, a_2, ..., a_{2k}, a_{\sigma(2k+1)}, \dots, a_{\sigma(n)}) , $$
where we think of the symmetric group $S_{n-2k}$ as the group of bijections of the set $\{2k+1, 2k+2, ..., n\}$ to itself.
Thus, by averaging, there exists $\sigma\in S_{n-2k}$ for which 
$$ \delta(q; a_1, a_2, ..., a_{2k}, a_{\sigma(2k+1)}, \dots, a_{\sigma(n)}) \leq \exp\left(-\frac{n \log(2A)}{4A \log q}\right) \frac{1}{n!} , $$
which gives the first part of Theorem \ref{mainbiasthm} on taking $C = 4A/\log(2A)$.

Similarly for the second part of the theorem, taking $k= \lceil n/2A \rceil$ we have
\begin{eqnarray}
&& \sum_{\sigma\in S_{n-2k}} \delta(q; a_1, a_3, ..., a_{2k-1}, a_{\sigma(2k+1)}, \dots, a_{\sigma(n)}, a_{2k}, ..., a_4, a_2) = \delta_{2k}^{\#}(q;a_1, \dots, a_n) \nonumber \\
& \geq & \exp\left(\frac{n \log(2A)}{4A \log q}\right) \frac{(n-2k)!}{n!} , \nonumber
\end{eqnarray}
so by averaging there exists $\sigma$ with $\delta(q; a_1, a_3, ..., a_{2k-1}, a_{\sigma(2k+1)}, \dots, a_{\sigma(n)}, a_{2k}, ..., a_4, a_2) \geq \exp\left(\frac{n \log(2A)}{4A \log q}\right) \frac{1}{n!}$.

Finally, if $\varphi(q)^{1/50} < n \leq \varphi(q)$ then set $m := \lfloor \varphi(q)^{1/50} \rfloor$ and assume that $q$ is so large that $m\ge 2$,
hence $m\ge \frac12 \phi(q)^{1/50}$.
As in the previous discussion there exists an $m$-tuple $(a_1,\dots,a_m)\in\mathcal{A}_m(q)$ for which
$$ \delta(q;a_1, \dots, a_m) \leq \exp\left(-\frac{\varphi(q)^{1/50}}{C \log q}\right) \frac{1}{m!} , $$
and there exists an $m$-tuple $(b_1,\dots,b_m)\in\mathcal{A}_m(q)$ for which
$$ \delta(q;b_1, \dots, b_m) \geq \exp\left(\frac{\varphi(q)^{1/50}}{C \log q}\right) \frac{1}{m!}, $$
with $C=8A/\log(2A)$.
Then if we choose any other coprime residues $a_{m+1}, ..., a_n$ mod $q$, we have
$$ \delta(q;a_1, \dots, a_m) = \sum_{\substack{\sigma \in S_{n} : \\ \sigma^{-1}(1) > \sigma^{-1}(2) > ... > \sigma^{-1}(m)}} \delta(q; a_{\sigma(1)}, a_{\sigma(2)}, \dots, a_{\sigma(n)}) . $$
There are $n!/m!$ terms in the sum, so it follows that for at least one permutation $\sigma$ we must have $\delta(q; a_{\sigma(1)}, a_{\sigma(2)}, \dots, a_{\sigma(n)}) \leq \exp\left(-\frac{\varphi(q)^{1/50}}{C \log q}\right) \frac{1}{n!}$, as desired. The analogous lower bound with the $b_i$ is proved exactly similarly.
\end{proof}

We conclude this introduction by discussing some of the ideas from our proofs. As we shall recall in section \ref{prelimsec}, under GRH and LI the logarithmic density $\delta(q;a_1, \dots, a_n)$ is the same as an ordering probability for certain random variables $X(q,a_1), ..., X(q,a_n)$. The ultimate source of the biases in our theorems is the fact that some of these random variables are not independent of one another, but have correlations of size $\xi \sim - \frac{\log 2}{\log q}$. Notice this is a completely different source of bias than the influence of being a quadratic residue or non-residue, which leads to Chebyshev's bias for small $q$ because the {\em mean values} of the corresponding $X(q,a_i)$ are slightly unequal.

As we mentioned, Harper and Lamzouri~\cite{HL} proved a non-uniformity result for the auxiliary quantities $\delta_k(q;a_1, \dots, a_n)$, and then deduced their non-uniformity result for $\delta(q;a_1, \dots, a_n)$ by averaging. The advantage of looking at $\delta_k(q;a_1, \dots, a_n)$ is that in the corresponding ordering probability, namely
$$ \p(X(q,a_1) > X(q,a_2) > ... > X(q,a_k) > \max_{k+1 \leq j \leq n} X(q,a_j)) , $$
the maximum of the $X(q,a_j)$ (when they are normalised by their standard deviations) is close to $\sqrt{2\log n}$ with very high probability. This means that $X(q,a_1), ..., X(q,a_k)$ are all larger than $\sqrt{2\log n}$ with very high probability, so if we can arrange to have negative correlations $\sim - \frac{\log 2}{\log q}$ between $\asymp k$ of those random variables, we obtain a bias $\asymp - \frac{1}{\log q} k (\sqrt{2\log n})^2 \asymp - \frac{k\log n}{\log q}$ in the exponent in the probability distribution function on that event.

Here we follow the same broad strategy, with two main changes. Firstly, in order to produce biases towards both large and small probabilities (unlike in \cite{HL} which only handled the small case) we work with the auxiliary density $\delta_{2k}^{\#}(q;a_1, \dots, a_n)$, in addition to $\delta_{2k}(q;a_1, \dots, a_n)$. In the ordering probability corresponding to $\delta_{2k}^{\#}(q;a_1, \dots, a_n)$, we have $X(q,a_1), X(q,a_3), ..., X(q,a_{2k-1}) > \sqrt{2\log n}$ and $X(q,a_2), X(q,a_4), ..., X(q,a_{2k}) < - \sqrt{2\log n}$ with very high probability (roughly speaking, when everything is normalised). So if we can arrange to have negative correlations $\sim - \frac{\log 2}{\log q}$ between pairs of the odd and even indexed $X(q,a_i)$, we obtain a positive bias $\asymp - \frac{1}{\log q} k \sqrt{2\log n} (-\sqrt{2\log n}) \asymp \frac{k\log n}{\log q}$ in the probability distribution function on that event. Secondly, whereas Harper and Lamzouri~\cite{HL} worked with fixed $k$, here we allow $k$ to be as large as a small constant times $n$. Thus, whereas Harper and Lamzouri obtain biases in the probability distribution of size at most $\asymp \frac{\log n}{\log q}$, we can obtain large biases of size $\asymp \frac{n}{\log q}$.

To implement the above strategy with $k$ large, there are two technical issues that must be overcome. Firstly, if we have arranged some correlations of size $\xi$ to produce a bias, we need to ensure that any other correlations do not interfere with that bias. Harper and Lamzouri~\cite{HL} did this using a ``large sieve'' kind of average estimate for correlations, but here we take a different approach by working with specially chosen residues $a_1, ..., a_{n}$ where the only correlations, except those of size $\xi$, are extremely small. This is perfectly acceptable if one only wants to exhibit some residues producing large biases, and makes the argument simpler (we don't need any sophisticated tools from Gaussian process theory) and possible to execute for very large $k$.

Secondly, in the above discussion and in our proofs we want to treat the $X(q,a_i)$ as jointly Gaussian random variables. As we shall discuss in section \ref{prelimsec}, in fact the $X(q,a_i)$ are a sum of independent random variables, and have variance $\asymp \varphi(q) \log q$, so we expect standard Berry--Esseen type ideas to produce a multivariate Gaussian approximation with an error term saving a polynomial in $q$, and with a polynomial dependence on the dimension $n$. Now comparing with Theorem \ref{mainbiasthm}, we are dealing with probabilities of size $1/n! = e^{-n\log n + O(n)}$ with $n$ as large as a power of $q$, which is much smaller than the possible Berry--Esseen saving. To get around this, we need a multivariate Gaussian approximation with a {\em relative} error term rather than an absolute one, which can therefore be useful even for very improbable events. This kind of multivariate ``moderate deviation'' result does exist in the probabilistic literature, but we were unable to locate a result that applied to our situation without quite a lot of reworking. (See Theorem 1 of Bentkus~\cite{bentkus} for a sample of what is available--- Bentkus obtains a moderate deviation Gaussian approximation for the norm of a sum of independent, identically distributed random elements in a Hilbert space.) Thus we prove our own bespoke result (see Proposition \ref{moddevprop} and Proposition \ref{normapproxpnr}, below), using a special construction of smooth test functions and an inductive ``replacement'' argument.

\subsection{Notation}

We adopt familiar order of magnitude notation of Bachmann--Landau, Vinogradov and Knuth.  The notations $f=O(g)$, 
$f\ll g$ and $g\gg f$ mean that  there exists a positive constant $C$
such that $|f| \le C g$ throughout the range of $f$ (either implicitly
understood or explicitly given).  The notations $f \order g$, $f= \Theta(g)$
mean that both $f\ll g$ and $f\gg g$ hold.  We have $f(x) \sim g(x)$ as 
$x\to a$ if $\lim_{x\to a} f(x)/g(x) = 1$, where $a$ may be finite, $\infty$
or $-\infty$.

We use $\PP$ and $\EE$ to denote probability and probabilistic expectation, respectively.  The underlying probability spaces will be described
explicitly or understood from context.

\section{Preliminaries on prime number races, and corresponding random variables}\label{prelimsec}
In this section, we review the connection between quantities like $\delta(q;a_1, \dots, a_n)$ and ordering probabilities for suitable random variables. See section 2 of Harper and Lamzouri~\cite{HL} for a similar but more detailed review of this material. We also show the existence of special sets $\cA$ of residues mod $q$, for which the correlation structure of the corresponding random variables has nice properties that we can exploit.

Given distinct reduced residues $a_1,\dots, a_n$ modulo $q$, we define
$$ E_{q;a_1,\dots,a_n}(x):=\Big(E(x;q,a_1), \dots, E(x;q,a_n)\Big),$$
where
$$ E(x;q,a):=\frac{\log x}{\sqrt{x}}\left(\varphi(q)\pi(x;q,a)-\pi(x)\right) , $$
and $\pi(x)$ denotes the total number of primes less than $x$. It turns out that the normalization is such that, if we assume GRH, $E_{q;a_1,\dots,a_n}(x)$ varies roughly boundedly as $x$ varies. Notice also that
$$ \pi(x;q,a_1) > \pi(x;q,a_2) > \cdots > \pi(x;q,a_n) \;\; \iff \;\; E(x;q,a_1) > E(x;q,a_2) > \cdots > E(x;q,a_n) . $$

Let us introduce the temporary notation
$$ \delta_{q;a_1,\dots,a_n}(S) := \lim_{X\to\infty} \frac{1}{\log X} \int_{\substack{x\in [2,X] \\ E_{q;a_1,\dots,a_n}(x)\in S}} \frac{dx}{x} . $$
The work of Rubinstein and Sarnak \cite{RuSa} shows, under GRH and LI, that for any Lebesgue measurable set $S\subset \mathbb{R}^n$ whose boundary has measure zero, the logarithmic density $\delta_{q;a_1,\dots,a_n}(S)$ exists.

Furthermore, Rubinstein and Sarnak provide a probabilistic characterisation of $\delta_{q;a_1,\dots,a_n}(S)$. For a non-principal Dirichlet character $\chi$ modulo $q$, denote by $\{\gamma_{\chi}\}$ the sequence of imaginary parts of the nontrivial zeros of  $L(s,\chi)$. If we assume LI then all of the non-negative values of $\gamma_{\chi}$ are linearly independent over $\QQ$, and in particular are distinct. Let $\chi_0$ denote the principal character modulo $q$ and define $\Gamma= \bigcup_{\chi\neq \chi_0\bmod q}\{\gamma_{\chi}\}$. Furthermore, let  $\{U(\gamma_{\chi})\}_{\gamma_{\chi}\in \Gamma, \gamma_\chi > 0}$ be a sequence of independent random variables uniformly distributed on the unit circle. Then we have
\begin{equation}\label{DensityMeasure}
\delta_{q;a_1,\dots,a_n}(S)= \int_{S} d\mu_{q;a_1,\dots,a_n},
\end{equation}
where $\mu_{q;a_1,\dots,a_n}$ is a certain probability measure
on $\RR^n$ (which is absolutely continuous when $n < \varphi(q)$).
Specifically, $\mu_{q;a_1,\dots,a_n}$ is
 the probability measure corresponding to the $\RR^{n}$-valued random vector $\big(X(q,a_1),\dots,X(q,a_n)\big)$,
where
\be\label{Xqa}
 X(q,a):= -C_q(a)+ \sum_{\substack{\chi\neq \chi_0\\ \chi\pmod q}}\re \left(2\chi(a)\sum_{\gamma_{\chi}>0} \frac{U(\gamma_{\chi})}{\sqrt{\frac14+\gamma_{\chi}^2}}\right),
 \ee
with 
\be\label{Cqa}
C_q(a):=-1+ |\{b \pmod q: b^2\equiv a \pmod q\}|.
\ee
It will turn out that the deterministic shifts $C_q(a)$ are sufficiently small that they can essentially be ignored when $q \rightarrow \infty$.

Let $\text{Cov}_{q;a_1,\dots,a_n}$ be the covariance matrix of $\big(X(q,a_1),\dots,X(q,a_n)\big)$. We need to understand this in order to understand the probabilities and dependencies of events involving the $X(q,a_i)$. Assuming GRH, one may compute (see section 2 of Harper and Lamzouri~\cite{HL} for fuller details and references) that
$$\textup{Cov}_{q;a_1,\dots,a_n}(i,j)= \begin{cases}  \var(q) &\text{ if } i=j \\ B_q(a_i,a_j) &\text{ if } i\neq j,\end{cases}$$
where
\begin{equation}\label{AsympVariance}
\var(q):=2\sum_{\substack{\chi\neq \chi_0\\ \chi\pmod q}}\sum_{\gamma_{\chi}>0}\frac{1}{\frac14+\gamma_{\chi}^2} \sim  \varphi(q)\log q \;\;\;\;\; \text{as} \; q \rightarrow \infty ,
\end{equation}
and 
\begin{equation}\label{BoundCovariance}
B_q(a,b):=\sum_{\substack{\chi\neq \chi_0 \\ \chi\pmod q}}\sum_{\gamma_{\chi}>0}\frac{\chi\left(\frac{b}{a}\right)+\chi\left(\frac{a}{b}\right)}{\frac14 +\gamma_{\chi}^2}\ll \varphi(q), \;\;\;\;\; a \neq b .
\end{equation}

In view of equations \eqref{AsympVariance} and \eqref{BoundCovariance}, we have the normalised correlation bound $\frac{B_q(a,b)}{\var(q)} \ll \frac{1}{\log q}$ for all distinct residues $a,b$ mod $q$. This bound can be attained by certain pairs $a,b$ (though only when $B_{q}(a,b)$ is negative), but otherwise it can be improved, depending on the arithmetic properties of $a,b$ and $q$. Understanding this in general is somewhat complicated, (see section 3 of Harper and Lamzouri~\cite{HL}, for example), but here we can afford to restrict to special sets of residues where the behaviour is nice. The following lemma will provide us with such sets.

\begin{lem}\label{correlation-2}
Assume GRH. Let $q$ be large and let $1\le k\le  n < q/(20\log^2 q)$.  Then there is a set $\cA=\{b_1,-b_1,b_2,-b_2,\ldots,b_k,-b_k,b_{k+1},\ldots,b_n\}$ of reduced residues modulo $q$, with $b_1,\ldots,b_{n}$ distinct integers in $\{1,\ldots,\fl{q/2} \}$,  and satisfying the correlation bounds
\[
 B_q(a,a') = -\varphi(q)(\log 2) \ell_q(a,a') + O (n\log^4 q)
\]
for every pair of distinct elements $a,a'\in \cA$, where $\ell_q(a,a')=1$ if $a+a' \equiv 0$ mod $q$, and $\ell_q(a,a')=0$ otherwise.
\end{lem}

\begin{proof}[Proof of Lemma \ref{correlation-2}]
Let $b_1, ..., b_n$ be any distinct primes in the interval $(5n\log^{2}q, 10n\log^{2}q]$ that do not divide $q$. By the prime number theorem, this interval contains $\sim \frac{5n\log^{2}q}{\log(10n\log^{2}q)} \geq 4n\log q$ primes, provided $q$ is large enough. And since $q$ has at most $\log q$ prime factors, if we remove all those we are still left with at least $3n\log q \geq n$ primes (say), so we can indeed choose $b_1, ..., b_n$ in this way.

The claimed estimate for $B_q(a,a')$ now follows from Proposition 6.1 of Lamzouri~\cite{L3}, noting that we have $1/2 < \frac{|a|}{|a'|} < 2$ for all $a,a'\in \cA$, so the term $\Lambda_{0}(\cdot)$ in that proposition always vanishes.
\end{proof}


\section{Gaussian approximation}\label{gaussiansec}
In this section, we prove results that will allow us to replace the actual random variables $X(q,a_i)$ arising from prime number races by Gaussian random variables with the same covariance structure. This will be important because jointly Gaussian random variables have an explicit probability density, depending only on the means and covariances, which we can work with to analyse the probabilities of events.

\subsection{Smooth test functions}\label{smoothtestsubsec}
We begin by constructing smooth weight functions that closely approximate indicators of the ``ordering'' events that we are interested in. Smooth weights will allow us to use Taylor expansion as a tool when we come to our Gaussian approximation.

\begin{lem}\label{smooth1d}
There exists an increasing function $\theta : \RR \rightarrow (0,1)$ that is three times continuously differentiable, and satisfies
$$ \theta(-x) \ll e^{-x} , \;\;\;\;\; 1 - \theta(x) \ll e^{-x} \;\;\;\;\; \forall x \geq 0 , $$
as well as
$$ |\frac{d^{m}}{d x^m} \theta(x)| \ll \theta(x) \;\;\; \forall 1 \leq m \leq 3, \; \forall x \in \RR . $$
\end{lem}

\begin{proof}[Proof of Lemma \ref{smooth1d}]
We take $\theta(x)$ to be the probability distribution function corresponding to a density proportional to $e^{-|x|}$, since repeatedly differentiating the exponential continues to yield an exponential.
Let $f(x)=e^{-|x|}$ for $|x|\ge 1$, and when $|x|<1$, let $f(x)$ be an
appropriate quadratic polynomial which ensures that $f\in C^2(\RR)$;
the polynomial $\frac{7-4x^2+x^4}{4e}$ satisfies these requirements.  
If  $m_0 := \int_{-\infty}^{\infty} f(t) dt$, then clearly
$$ \theta(x) := \frac{1}{m_0} \int_{-\infty}^{x} f(t) dt $$
satisfies all the required properties of the lemma.  In particular,
for $x\le -1$, we have $|\theta^{(j)}(x)| = \theta(x)= \frac{1}{m_0} e^x$ 
for $j=1,2,3$
and for $x\ge -1$, $\theta(x) \ge 1/(em_0) \gg |\theta^{(j)}(x)|$
for $j=1,2,3$.
\end{proof}

By taking a product of (shifted and dilated) copies of the function $\theta(x)$, we can obtain multi-dimensional smooth test functions, which is what we shall
 actually need. We shall use the following notation: let $S \subseteq \{1,2,...,n\} \times \{1,2,...,n\}$ and let
$$ R(S) := \{(x_1 , ..., x_n) \in \RR^n : x_i > x_j \; \forall (i,j) \in S\} . $$
We denote by $\textbf{1}_{R(S)}(x_1, ..., x_n)$ the indicator function of the set $R(S)$.  We say that the set $S$ is \emph{admissible} if $R(S)$ is nonempty.
In particular, $S$ does not contain any diagonal pairs $(i,i)$.
For an admissible set $S$ and $\delta>0$ we also define the two functions
$h_{S,\delta}^+$ and  $h_{S,\delta}^-$ by
\be\label{hg}
 h_{S,\delta}^{\pm}(x_1,...,x_n) := \prod_{(i,j) \in S} g(x_i - x_j \pm \sqrt{\delta}) , \;\;\;\;\; \text{where} \; g(x) := \theta(x/\delta) .
 \ee
We next show that these functions are good approximations of
 $\textbf{1}_{R(S)}(x_1, ..., x_n)$.

\begin{pro}\label{smoothmultid}
Let $S$ be an admissible set.  Each of $h_{S,\delta}^-$ and $h_{S,\delta}^+$ is a three times continuously differentiable function from $\RR^{n}$ to $[0,\infty)$, satisfying
$$ h_{S,\delta}^{\pm}(\textbf{x}) \leq 1 , \;\;\; \sup_{1 \leq i \leq n} \Big|\frac{\partial}{\partial x_i} h_{S,\delta}^{\pm}(\textbf{x})\Big| \ll \frac{n}{\delta} h_{S,\delta}^{\pm}(\textbf{x}) , \;\;\; \sup_{1 \leq i,j,k \leq n} \Big|\frac{\partial^{3}}{\partial x_i \partial x_j \partial x_k} h_{S,\delta}^{\pm}(\textbf{x})\Big| \ll \pfrac{n}{\delta}^3 h_{S,\delta}^{\pm}(\textbf{x}).$$
Furthermore, provided that $e^{-1/\sqrt{\delta}} \leq 1/n^2$, we have
 $$ h_{S,\delta}^{-}(x_1,...,x_n) - O(e^{-1/\sqrt{\delta}}) \leq \textbf{1}_{R(S)}(x_1, ..., x_n) \leq h_{S,\delta}^{+}(x_1,...,x_n) + O(n^2 e^{-1/\sqrt{\delta}}). $$
\end{pro}

An important feature of this result is the fact that the derivatives of $h_{S,\delta}^{\pm}$ are always controlled by $h_{S,\delta}^{\pm}$ itself, even at points where $h_{S,\delta}^{\pm}$ is very small. We will exploit this later in our Gaussian approximation.

\begin{proof}[Proof of Proposition \ref{smoothmultid}]
The partial derivative bounds in Proposition \ref{smoothmultid} follow exactly as in Lemma 4.3 of Harper and Lamzouri~\cite{HL}, for example, by repeated application of the product rule together with the fact that $\frac{d^{m}}{d x^m} g(x) = (1/\delta)^m \frac{d^{m}}{d x^m} \theta(x/\delta) \ll (1/\delta)^m \theta(x/\delta) = (1/\delta)^m g(x)$ for $1 \leq m \leq 3$.

For the lower bound on $\textbf{1}_{R(S)}$, it will suffice to show that whenever $(x_1, ..., x_n) \notin R(S)$ we have a uniform upper bound
$$ h_{S,\delta}^{-}(x_1,...,x_n) \ll e^{-1/\sqrt{\delta}} . $$
Indeed, if $(x_1, ..., x_n) \notin R(S)$ then for at least one pair $(i,j) \in S$, we have $x_i - x_j - \sqrt{\delta} < -\sqrt{\delta}$, and therefore
 $$ h_{S,\delta}^{-}(x_1,...,x_n) \leq g(-\sqrt{\delta}) = \theta(-1/\sqrt{\delta}) \ll e^{-1/\sqrt{\delta}} . $$

Similarly, for the upper bound it will suffice to show that whenever $(x_1, ..., x_n) \in R(S)$ we have
 $$ h_{S,\delta}^{+}(x_1,...,x_n) = 1 - O(n^2 e^{-1/\sqrt{\delta}}) . $$
Indeed, if $(x_1, ..., x_n) \in R(S)$ then we have $x_i - x_j + \sqrt{\delta} \geq \sqrt{\delta}$ for each pair $(i,j) \in S$, and so
 $$ h_{S,\delta}^{+}(x_1,...,x_n) \geq g(\sqrt{\delta})^{\#S} \geq \theta(1/\sqrt{\delta})^{n^2} , $$
 which implies the result since we have $\theta(1/\sqrt{\delta}) = 1 - O(e^{-1/\sqrt{\delta}})$.

\end{proof}

\subsection{A Lindeberg type argument}
In this subsection, we shall establish a ``moderate deviation'' type of Gaussian approximation result relative to smooth test functions $h$. By ``moderate deviation'', we mean that we want the theorem to remain useful even in the tails of the Gaussian, where an approximation with a simple absolute error term would not be useful. Instead, we are seeking a result with a relative error (together with an absolute error term that is extremely small).

We shall obtain our Gaussian approximation with a version of the Lindeberg replacement strategy, which was originally used to prove the classical central limit theorem for sums of independent random variables. The idea is to Taylor expand the test function $h$ to third order, and replace the independent summands $\X^{(j)}$ by corresponding Gaussian random variables $\Z^{(j)}$ one at a time. The slightly non-standard assumptions we shall make on $h$, that its partial derivatives are controlled by $h$ itself, will allow us to make the biggest error term a relative rather than absolute one. (So far as we are aware, the use of such non-standard $h$ is novel in this context. The use of a Lindeberg type method is certainly not novel, for example this is how Bentkus~\cite{bentkus} proceeds.)

Let $C_1 , C_3 \geq 1$, and let $h : \RR^n \rightarrow [0,\infty)$ be a three times continuously differentiable function, such that for all $\textbf{x} \in \RR^n$ we have
$$ h(\textbf{x}) \leq 1 , \;\;\; \sup_{1 \leq i \leq n} \left|\frac{\partial}{\partial x_i} h(\textbf{x})\right| \leq C_1 h(\textbf{x}) , \;\;\; \sup_{1 \leq i,j,k \leq n} \left|\frac{\partial^{3}}{\partial x_i \partial x_j \partial x_k} h(\textbf{x})\right| \leq C_3 h(\textbf{x}) . $$

\begin{lem}\label{lindeberglemma}
Let $0 < \epsilon \leq \min\{1/(2C_1) , 1/(2 C_{3}^{1/3}) \}$.
Then uniformly for any fixed $\Y \in \RR^n$, any real random vector $\X=(X_1 , ..., X_n)$ whose components have mean zero, we have the following. 
If $\Z=(Z_1 , ..., Z_n)$ is a multivariate normal random vector whose components have mean zero and the same covariances $r_{i,j} := \ex X_i X_j $ as $\X$, then
\begin{eqnarray}
\ex h(\Y+\X) & = & (1 + \gamma(\epsilon))\ex h(\Y+\Z) + \nonumber \\
&& + O\Biggl(C_3 \Biggl(\sum_{i=1}^{n} \ex\Big(\textbf{1}_{|X_i| > \epsilon/n} (\epsilon^{3} + n^{2} |X_i|^3)\Big)+ \sum_{i=1}^{n}(\epsilon^3 + n^2 r_{i,i}^{3/2}) e^{-\epsilon^{2}/(2n^2 r_{i,i})} \Biggr) \Biggr) , \nonumber
\end{eqnarray}
where $\gamma(\epsilon) = \gamma(\epsilon,h,\Y,\X)$ satisfies $|\gamma(\epsilon)| \leq 6C_3 \epsilon^3$.
\end{lem}

\begin{proof}[Proof of Lemma \ref{lindeberglemma}]
By the multivariate form of Taylor's theorem, we have
$$ h(\Y+\X) = h(\Y) + \sum_{i=1}^{n} X_i \frac{\partial}{\partial x_i} h(\Y) + \sideset{}{^*}{\sum}_{i,j = 1}^{n} X_i X_j \frac{\partial^{2}}{\partial x_i \partial x_j} h(\Y) + R(h,\Y,\X) , $$
where the $*$ on the sum indicates that the terms $i=j$ should be counted with weight $1/2$, and where the error term satisfies
\begin{eqnarray}
|R(h,\Y,\X)| & \leq & \sup_{\theta \in [0,1]} \sup_{1 \leq i,j,k \leq n} \left|\frac{\partial^{3}}{\partial x_i \partial x_j \partial x_k} h(\Y + \theta \X)\right| \Big(\sum_{i=1}^{n} |X_i|\Big)^3 \nonumber \\
& \leq & C_3 \sup_{\theta \in [0,1]} h(\Y + \theta \X) \Big(\sum_{i=1}^{n} |X_i|\Big)^3 . \nonumber
\end{eqnarray}
One has the same expansion for $h(\Y+\Z)$. Taking expectations and then taking the difference $\ex h(\Y+\X) - \ex h(\Y+\Z)$, all of the main terms cancel because we assume $\Z$ and $\X$ have the same means and covariances, so we get
\begin{eqnarray}
&& |\ex h(\Y+\X) - \ex h(\Y+\Z)| \leq \ex|R(h,\Y,\X)| + \ex|R(h,\Y,\Z)| \nonumber \\
& \leq & C_3 \Biggl(\ex \left\{\sup_{\theta \in [0,1]} h(\Y + \theta \X) \Big(\sum_{i=1}^{n} |X_i|\Big)^3\right\} + \ex \left\{\sup_{\theta \in [0,1]} h(\Y + \theta \Z) \Big(\sum_{i=1}^{n} |Z_i|\Big)^3\right\}\Biggr) . \nonumber
\end{eqnarray}

Now using Taylor's theorem again, for any $\theta \in [0,1]$ we have
\begin{eqnarray}
|h(\Y + \theta \X) - h(\Y+\X)| & \leq & \sup_{\phi \in [0,1]} \sup_{1 \leq i \leq n} \left|\frac{\partial}{\partial x_i} h(\Y + \phi \X)\right| \sum_{i=1}^{n} |X_i| \nonumber \\
& \leq & C_1 \sup_{\phi \in [0,1]} h(\Y + \phi \X) \sum_{i=1}^{n} |X_i| . \nonumber
\end{eqnarray}
In particular, if we happen to have $\sum_{i=1}^{n} |X_i| \leq \epsilon \leq 1/(2C_1)$ then it follows that $$\sup_{\theta \in [0,1]} h(\Y + \theta \X) \leq 2 h(\Y+\X).$$ So we always have the bound
$$ \ex \sup_{\theta \in [0,1]} h(\Y + \theta \X) \Big(\sum_{i=1}^{n} |X_i|\Big)^3 \leq 2\epsilon^3 \ex h(\Y+\X) + \ex \textbf{1}_{\sum_{i=1}^{n} |X_i| > \epsilon} \Big(\sum_{i=1}^{n} |X_i|\Big)^3 . $$
The same argument applies to $\sup_{\theta \in [0,1]} h(\Y + \theta \Z)$.

Putting everything together, we get
\begin{multline*}
\big|\ex h(\Y+\X) - \ex h(\Y+\Z)\big|  \leq  2C_3 \epsilon^{3} \big(\ex h(\Y+\X) + \ex h(\Y+\Z)\big)+  \\
+ C_3 \left(\ex \Bigg(\textbf{1}_{\sum_{i=1}^{n} |X_i| > \epsilon} \Big(\sum_{i=1}^{n} |X_i|\Big)^3\Bigg) + \ex \Bigg(\textbf{1}_{\sum_{i=1}^{n} |Z_i| > \epsilon} \Big(\sum_{i=1}^{n} |Z_i|\Big)^3\Bigg)\right) .
\end{multline*}
Furthermore, our assumptions on $\epsilon$ imply that $2C_3 \epsilon^3 \leq 1/4$, and so
$$ \frac{1+2C_3 \epsilon^3}{1-2C_3 \epsilon^3} = 1 + \frac{4C_3 \epsilon^3}{1-2C_3 \epsilon^3} \leq 1 + 6C_3 \epsilon^3 , $$
and similarly $\frac{1-2C_3 \epsilon^3}{1+2C_3 \epsilon^3} \geq 1 - 4C_3 \epsilon^3$. So rearranging our above bound, we find that
\begin{multline*}
\ex h(\Y+\X) = (1 + \gamma(\epsilon))\ex h(\Y+\Z) + \\
 O\Biggl(C_3 \Biggl(\ex\Big(\textbf{1}_{\sum_{i=1}^{n} |X_i| > \epsilon} \big(\sum_{i=1}^{n} |X_i|\big)^3 \Big)+ \ex \Big(\textbf{1}_{\sum_{i=1}^{n} |Z_i| > \epsilon} (\sum_{i=1}^{n} |Z_i|)^3\Big) \Biggr) \Biggr),
 \end{multline*}
where $|\gamma(\epsilon)| \leq 6C_3 \epsilon^3$.

Next we want to work with the ``big Oh'' term a little, to replace terms depending on $\sum_{i=1}^{n} |X_i|$ and $\sum_{i=1}^{n} |Z_i|$ by terms that only require information about individual components $X_i$ and $Z_i$, as in the statement
 of the lemma.   First, H\"older's inequality implies that $(\sum_{i=1}^{n} |X_i|)^3 \le n^2 \sum_{i=1}^{n} |X_i|^3$. Furthermore, by splitting into cases according as each $|X_i| > \epsilon/n$ or not, we get
\begin{eqnarray}
n^2 \ex \Big(\textbf{1}_{\sum_{i=1}^{n} |X_i| > \epsilon} \sum_{i=1}^{n} |X_i|^3\Big) & \leq & n^2 \ex \left(\textbf{1}_{\sum_{i=1}^{n} |X_i| > \epsilon} \sum_{i=1}^{n} \left((\epsilon/n)^3 + \textbf{1}_{|X_i| > \epsilon/n} |X_i|^3\right)\right) \nonumber \\
& \leq & n^{2} \sum_{i=1}^{n} \ex \Big(\textbf{1}_{|X_i| > \epsilon/n} |X_i|^3\Big) + \epsilon^3 \ex \big( \textbf{1}_{\sum_{i=1}^{n} |X_i| > \epsilon}\big)  . \nonumber
\end{eqnarray}
And we can bound this further using the inequality 
$$\ex \textbf{1}_{\sum_{i=1}^{n} |X_i| > \epsilon} \leq \ex \textbf{1}_{\max_{i} |X_i| > \epsilon/n} \leq \sum_{i=1}^{n} \ex \textbf{1}_{|X_i| > \epsilon/n},$$
 so overall we have
$$ \ex \textbf{1}_{\sum_{i=1}^{n} |X_i| > \epsilon} (\sum_{i=1}^{n} |X_i|)^3 \le n^2 \ex \textbf{1}_{\sum_{i=1}^{n} |X_i| > \epsilon} \sum_{i=1}^{n} |X_i|^3 \leq \sum_{i=1}^{n} \ex \textbf{1}_{|X_i| > \epsilon/n} (\epsilon^{3} + n^{2} |X_i|^3) . $$
We have the analogous bound for the term involving the $Z_i$.

Finally, in the case of the $Z_i$, since we know that $Z_i \sim N(0, r_{i,i})$ we can further say that $\ex (\textbf{1}_{|Z_i| > \epsilon/n}) = \p(|Z_i| > \epsilon/n) \ll e^{-\epsilon^{2}/(2n^2 r_{i,i})}$, and that
$$ \ex \left(\textbf{1}_{|Z_i| > \epsilon/n} |Z_i|^3\right) \ll r_{i,i}^{3/2}\left(1 + (\epsilon/n\sqrt{r_{i,i}})^3\right) e^{-\epsilon^{2}/(2n^2 r_{i,i})} \ll \left(r_{i,i}^{3/2} + \frac{\epsilon^3}{n^3}\right) e^{-\epsilon^{2}/(2n^2 r_{i,i})} . $$
Inserting these estimates in the ``big Oh'' term completes the proof.
\end{proof}

Applying Lemma \ref{lindeberglemma} inductively, we shall prove our Gaussian approximation result for sums of $m$ independent random vectors.

\begin{pro}\label{moddevprop}
For $1 \leq j \leq m$, let $\X^{(j)} = (X_{1}^{(j)},...,X_{n}^{(j)})$ be independent $\RR^n$-valued random vectors whose components have mean zero, and let $\Z^{(j)}$, $1\le j\le m$, be independent multivariate normal random vectors whose components have mean zero and the same covariances $r(j)_{i,k} = \ex X^{(j)}_{i} X^{(j)}_{k}$ as $\X^{(j)}$.

Let $0 < \epsilon \leq \min\{1/(2C_1) , 1/(3 C_{3}^{1/3} m^{1/3}) \}$ be a small parameter.

Then we have
\begin{eqnarray}
\ex h\big(\sum_{j=1}^{m} \X^{(j)}\big) & = & e^{\Delta(\epsilon)} \ex h\big(\sum_{j=1}^{m} \Z^{(j)}\big) + O\Biggl( C_3 \sum_{j=1}^{m} \sum_{i=1}^{n} \ex \Big(\textbf{1}_{|X^{(j)}_i| > \epsilon/n} \big(\epsilon^{3} + n^{2} |X^{(j)}_i|^3\big)\Big) \Biggr) \nonumber \\
&& + O\Biggl(C_3 \sum_{j=1}^{m} \sum_{i=1}^{n} (\epsilon^3 + n^2 r(j)_{i,i}^{3/2}) e^{-\epsilon^{2}/(2n^2 r(j)_{i,i})} \Biggr) , \nonumber
\end{eqnarray}
where $\Delta(\epsilon) = \Delta_m(\epsilon,h,\{\X^{(j)}\})$ satisfies $|\Delta(\epsilon)| \leq 12mC_3 \epsilon^3$.
\end{pro}

To get an idea of the potential usefulness of this result, the reader might consider the case where all components $X^{(j)}_{i}$ have variance $\asymp 1/m$ (and so all components of the sum $\sum_{j=1}^{m} \X^{(j)}$ have variance $\asymp 1$), and $C_1, C_3, n$ are fairly small compared with $m$. Because $\Delta(\epsilon)$ decays cubically with $\epsilon$, any choice of $\epsilon$ that is rather smaller than $1/(mC_3)^{1/3}$ will make the relative error term $e^{\Delta(\epsilon)}$ close to 1. Meanwhile, if the random components $X^{(j)}_{i}$ are reasonably concentrated on the order of their standard deviations, we expect any choice of $\epsilon$ rather larger than $n/\sqrt{m}$ to yield substantial savings in the ``big Oh'' terms. So we have room to make a choice of $\epsilon$ that simultaneously controls all these terms.

\begin{proof}[Proof of Proposition \ref{moddevprop}]
We proceed by induction on $m$. Our inductive hypothesis will be the estimate stated in the proposition, with an additional multiplier $e^{6C_3 \epsilon^{3} (m-1)}$ in the ``big Oh'' terms. If we can prove this inductively we will be done, since our conditions on $m$ imply this factor is $\leq e^{6/27} \ll 1$.

When $m=1$, Proposition \ref{moddevprop} is a direct consequence of Lemma \ref{lindeberglemma}, with $\Y$ taken as the 0 vector.

For the inductive step, if we first take $\Y = \sum_{j=1}^{m-1} \X^{(j)}$ and condition on the value of $\Y$ (which is independent of $\X^{(m)}$ and $\Z^{(m)}$), then Lemma \ref{lindeberglemma} implies that
\begin{eqnarray}
\ex h(\Y+\X^{(m)}) & = & (1 + \gamma(\epsilon))\ex h(\Y+\Z^{(m)}) + O\Biggl(C_3 \sum_{i=1}^{n} \ex \Big(\textbf{1}_{|X^{(m)}_i| > \epsilon/n} (\epsilon^{3} + n^{2} |X^{(m)}_i|^3)\Big) \Biggr) \nonumber \\
&& + O\Biggl(C_3 \sum_{i=1}^{n}(\epsilon^3 + n^2 r(m)_{i,i}^{3/2}) e^{-\epsilon^{2}/(2n^2 r(m)_{i,i})} \Biggr) . \nonumber
\end{eqnarray}
Now if we condition on the value of $\Z^{(m)}$, and apply the inductive hypothesis with $h(\cdot)$ replaced by $h(\cdot + \Z^{(m)})$ (which obeys all the same partial derivative bounds), we get that $\ex h(\sum_{j=1}^{m-1} \X^{(j)}+\Z^{(m)})$ is
\begin{eqnarray}
&&  e^{\Delta_{m-1}(\epsilon)} \ex h(\sum_{j=1}^{m-1} \Z^{(j)}+\Z^{(m)}) + O\Biggl(e^{6C_3 \epsilon^{3} (m-1)} C_3 \sum_{j=1}^{m-1} \sum_{i=1}^{n} \ex \textbf{1}_{|X^{(j)}_i| > \epsilon/n} (\epsilon^{3} + n^{2} |X^{(j)}_i|^3) \Biggr) \nonumber \\
&& + O\Biggl(e^{6C_3 \epsilon^{3} (m-1)} C_3 \sum_{j=1}^{m-1} \sum_{i=1}^{n}(\epsilon^3 + n^2 r(j)_{i,i}^{3/2}) e^{-\epsilon^{2}/(2n^2 r(j)_{i,i})} \Biggr) . \nonumber
\end{eqnarray}
The above is then multiplied by $(1 + \gamma(\epsilon))$ in the expression for
$\ex h(\Y+\X^{(m)})$.
Using the fact that $$e^{-12C_3 \epsilon^3} \leq e^{-2|\gamma(\epsilon)|} \leq (1+\gamma(\epsilon)) \leq e^{|\gamma(\epsilon)|} \leq e^{6C_3 \epsilon^{3}},$$ we complete the induction.
\end{proof}

\subsection{Application to prime number races}
In this subsection, we specialise the discussion in the preceding propositions to the case of prime number races. Recall the random variables
$$ X(q,a) = -C_q(a)+ \sum_{\substack{\chi\neq \chi_0\\ \chi\pmod q}}\re \left(2\chi(a)\sum_{\gamma_{\chi}>0} \frac{U(\gamma_{\chi})}{\sqrt{\frac14+\gamma_{\chi}^2}}\right) $$
from section \ref{prelimsec}, and recall the sets $R(S) = \{(x_1 , ..., x_n) \in \RR^n : x_i \geq x_j \; \forall (i,j) \in S\} \subseteq \RR^n$ and the corresponding smooth test functions $h_{S,\delta}^{-}, h_{S,\delta}^{+}$ from section \ref{smoothtestsubsec}.

\begin{pro}\label{normapproxpnr}
Let $q$ be large, and suppose that $2 \leq n \leq \varphi(q)^{1/12}$.

Let $\X := \left(\frac{X(q,a_i)}{\sqrt{\var(q)}}\right)_{1 \leq i \leq n}$, where $a_1, ..., a_n$ are coprime residue classes mod $q$, and let $\Z=(Z_{i})_{1 \leq i \leq n}$ denote a multivariate normal random vector whose components have mean zero, variance one, and correlations $\ex Z_{i}Z_{j} := \frac{B_{q}(a_i,a_j)}{\var(q)}$.

Then for any small parameter $(\frac{n^{5} \log q}{\sqrt{\varphi(q)}})^{1/3} \leq \delta \leq \frac{1}{\log^{2}n}$ and any admissible set $S$, we have
\begin{multline*}
 \(1 + O\pfrac{1}{\log q}\)\ex h^{-}_{S,\delta}(\Z) - O\(e^{-2/\sqrt{\delta}} + n e^{- \Theta\big(\frac{\delta^{2} \varphi(q)^{1/3}}{n^4 \log^{2/3}q}\big)}\)  \\
\leq  \p(\X \in R(S)) \leq \(1 + O\pfrac{1}{\log q}\)\ex h^{+}_{S,\delta}(\Z) + O\(n^2 e^{-2/\sqrt{\delta}} + n e^{- \Theta\big(\frac{\delta^{2} \varphi(q)^{1/3}}{n^4 \log^{2/3}q}\big)}\) . \nonumber
\end{multline*}
\end{pro}

We remark that the condition $n \leq \varphi(q)^{1/12}$ is stronger than
necessary, but is convenient at one point in the proof, and without it the proposition is trivial because the ``big Oh'' term will not be less than 1.

\begin{proof}[Proof of Proposition \ref{normapproxpnr}]
In the first place, the condition $\delta \leq \frac{1}{\log^{2}n}$ means that Proposition \ref{smoothmultid} is applicable with $\delta$ replaced by $\delta/4$, and gives
$$ \ex h^{-}_{S,\delta/4}(\X) - O(e^{-2/\sqrt{\delta}}) \leq \p(\X \in R(S)) \leq \ex h^{+}_{S,\delta/4}(\X) + O(n^2 e^{-2/\sqrt{\delta}}) . $$
Now the shifts $D_i :=\frac{C_q(a_i)}{\sqrt{\var(q)}}$ in the components $X_i:=\frac{X(q,a_i)}{\sqrt{\var(q)}}$ (see \eqref{Xqa}) are a little awkward, since they cause the components to have non-zero mean. However, recalling our expressions for $C_q(a)$ \eqref{Cqa} and $\var(q)$ \eqref{AsympVariance}, together with the lower bound $\delta \geq (\frac{n^{5} \log q}{\sqrt{\varphi(q)}})^{1/3} \geq \frac{1}{\varphi(q)^{1/6}}$, we find that
$$ |D_i| \leq \frac{\sum_{m|q} 1}{\sqrt{\var(q)}} \ll \frac{1}{\varphi(q)^{0.49}} \leq \frac{\sqrt{\delta}}{4}. $$
It follows that for any pair $(i,j)$ we have
\[
X_i - X_j + \sqrt{\delta/4} \le \widetilde{X}_i - \widetilde{X}_j + \sqrt{\delta}
\]
and
\[
X_i - X_j - \sqrt{\delta/4} \ge \widetilde{X}_i - \widetilde{X}_j - \sqrt{\delta},
\]
where
\[
\widetilde{X}_i = X_i + D_i \qquad (1\le i\le n).
\]
So recalling the definitions \eqref{hg} of $h^{+}_{S,\delta/4}, h^{-}_{S,\delta/4}$, we may remove the shifts $D_i$ and still have the same upper and lower bounds for $\p(\X \in R(S))$, at the cost of replacing $h^{\pm}_{S,\delta/4}$ by $h^{\pm}_{S,\delta}$.
That is, we have
\[
\ex h^{-}_{S,\delta}(\widetilde{\X}) - O(e^{-2/\sqrt{\delta}}) \leq \p(\X \in R(S)) \leq \ex h^{+}_{S,\delta}(\widetilde{\X}) + O(n^2 e^{-2/\sqrt{\delta}}) . 
\]

Now we want to show that $\ex h^{+}_{S,\delta}(\widetilde{\X})$ may be replaced by $\ex h^{+}_{S,\delta}(\Z)$, and $\ex h^{-}_{S,\delta}(\widetilde{\X})$ by $\ex h^{-}_{S,\delta}(\Z)$, up to acceptable error terms. We apply Proposition \ref{moddevprop} with the sum over $1 \leq j \leq m$ replaced by a sum over characters $\chi \neq \chi_{0}$ mod $q$ (so $m = \varphi(q) - 1$), and with $\widetilde{X}_j$ replaced by
$$ \widetilde{X}^{(\chi)} := \left( \frac{1}{\sqrt{\var(q)}} \re \left(2\chi(a_i)\sum_{\gamma_{\chi}>0} \frac{U(\gamma_{\chi})}{\sqrt{\frac14+\gamma_{\chi}^2}}\right) \right)_{1 \leq i \leq n} . $$
Notice that these are indeed independent $\RR^n$-valued random vectors whose components have mean zero, since the underlying random variables $U(\gamma_{\chi})$ are independent and have mean zero. For the test functions $h_{S,\delta}^{\pm}$, we may take $C_1 \asymp n/\delta$ and $C_3 \asymp (n/\delta)^3$, so in Proposition \ref{moddevprop} we are permitted to make any choice of $0 < \epsilon \le \delta/(n \varphi(q)^{1/3})$.

Turning to the error terms in Proposition \ref{moddevprop}, for each $1 \leq i \leq n$ and each $\chi \neq \chi_{0}$ we have
$$ r(\chi)_{i,i} = \frac{1}{\var(q)} \sum_{\gamma_{\chi}>0} \frac{\ex (\re 2\chi(a_i) U(\gamma_{\chi}))^2}{\frac14+\gamma_{\chi}^2} = \frac{2}{\var(q)} \sum_{\gamma_{\chi}>0} \frac{1}{\frac14+\gamma_{\chi}^2} \ll \frac{1}{\varphi(q)} , $$
where the final inequality uses the standard estimate $\sum_{\gamma_{\chi}>0} \frac{1}{\frac14+\gamma_{\chi}^2} \ll \log q$ (see Corollary 10.18 of Montgomery and Vaughan~\cite{mv}, for example). Furthermore, an exponential moment calculation (as in the proof of Lemma 2.3 of Lamzouri~\cite{La1}) implies that $$\p(|\widetilde{X}^{(\chi)}_i| > r) \ll e^{-\Theta(r^2 \varphi(q))}$$ for each $1 \leq i \leq n$, each $\chi \neq \chi_{0}$ and any $r \geq 0$. This simply says that, as we might expect, the components $\widetilde{X}^{(\chi)}_i$ have Gaussian-type tails.  A consequence of this bound is
\[
\ex \textbf{1}_{|\widetilde{X}_i^{\chi}|>r} |\widetilde{X}_i^{\chi}|^3 \ll (r^3 + \frac{1}{\varphi(q)^{3/2}}) e^{-\Theta(r^2 \varphi(q))}.
\]
 Inserting all this information in Proposition \ref{moddevprop}, we get for any $0 < \epsilon \le \delta/(n \varphi(q)^{1/3})$ that
$$ \ex h^{\pm}_{S,\delta}(\widetilde{\X}) = e^{\Delta(\epsilon)} \ex h^{\pm}_{S,\delta}(\Z) + O\Biggl( (\frac{n}{\delta})^3 \varphi(q) n \biggl(\epsilon^3 + \frac{n^2}{\varphi(q)^{3/2}} \biggr) e^{- \Theta(\epsilon^{2} \varphi(q)/n^2)} \Biggr) , $$
where $|\Delta(\epsilon)| \ll \varphi(q) (n\epsilon/\delta)^3$.

Finally, if we take $\epsilon = \frac{\delta}{n \varphi(q)^{1/3} \log^{1/3}q}$ then, in view of the condition $\delta \geq (\frac{n^{5} \log q}{\sqrt{\varphi(q)}})^{1/3}$, we have $\epsilon^3 \geq \frac{n^2}{\varphi(q)^{3/2}}$. Then a quick calculation verifies that our ``big Oh'' term is of the form claimed in the proposition.
\end{proof}

\section{Proof of Theorem \ref{kExtremeBias}: an expression for the density}\label{densitysec}
Let $q$ be large. In this section, we can work under the relatively weak hypotheses that 
\be\label{kn_sec4}
1 \leq k \leq n/2 \leq \sqrt{\frac{\varphi(q)}{\log^{5}q}}.
\ee
 We shall consider the tuple of residues $$(a_1, \dots, a_n)= (b_1, -b_1, \dots, b_k, -b_k, b_{k+1}, \dots, b_{n-k}),$$ where the $b_i$'s satisfy the conclusion of Lemma \ref{correlation-2}. Our goal is to calculate explicit expressions for the probability density function of the vector $\Z = (Z_i)_{1 \leq i \leq n}$, where the $Z_i$ are the jointly Gaussian random variables from Proposition \ref{normapproxpnr} that correspond to the residues $a_i$.

The covariance matrix of $\Z$ is $C=C(\Z)=A+E$, where
\[
 A = \begin{bmatrix} 
        1 & \xi &  &  &  &  &  &  &  & &  &  \\
        \xi & 1  &  &   &  &  &  &  &  & &  & \\
         &  & 1 & \xi  &  &  &  &  &  &  &  & \\ 
         &  & \xi & 1  &  &  &  &  &  & &  & \\ 
          &  &  &  & \ddots &  &  &  & &  & & \\
          &  &  &  &  & 1 & \xi &  &  & &  & \\
          &  &  &  &  & \xi & 1 &  &  &  &  & \\
          &  &  &  &  &  &  & 1 &  & &  & \\
             &  &  &  &  &  &  &  & 1 &  &  & \\
          &  &  &  &  &  & &  &  & \ddots  &  &  \\
          &  &  &  &  &  &  &  &  & & 1 & \\
                     &  &  &  &  &  &  &  &  &  &  & 1
     \end{bmatrix}
\]
with
\begin{equation}\label{XI}
\xi := - \frac{\varphi(q)\log 2}{\Var(q)} \sim - \frac{\log 2}{\log q} ,
\end{equation}
 and all of the entries of $E$ are 
\emph{uniformly small}, in fact bounded by  $\eps$, where
\be\label{epsilon}
\eps\ll \frac{n \log^{4}q}{\Var(q)} \ll \frac{n\log^{3}q}{\varphi(q)}\ll \frac{\sqrt{\log q}}{\sqrt{\varphi(q)}} .
\ee

Let $f(x_1, \cdots, x_n)$ be the joint density function of $Z_1, \dots, Z_n$. Since the $Z_i$ are jointly Gaussian, we have
\begin{equation}\label{PDFExplicit}
 f(x_1, \cdots, x_n)=\frac{1}{(2\pi)^{n/2}  \sqrt{\det C}} \exp \Big\{ -\frac12 \xx^T C^{-1} \xx \Big\}.
\end{equation}
\newcommand{\norm}[1]{\| #1 \|}
For $\x=(x_1, \dots, x_n)\in \mathbb{R}^n$, we denote by $\norm{\x_{2k}}=(x_1^2+\cdots+x_{2k}^2)^{1/2}$ the Euclidean norm of the first $2k$ components of $\x$, and by $\norm{\x^{n-2k}}=(x_{2k+1}^2+\cdots+x_{n}^2)^{1/2}$ the Euclidean norm of the last $n-2k$ components of $\x$, and by $\norm{\x}=(x_{1}^2+\cdots+x_{n}^2)^{1/2}$ the Euclidean norm of the full vector $\x$. 
We prove the following result.
\begin{pro}\label{ASYMPDF} Let $q$ be sufficiently large and assume \eqref{kn_sec4}. Let $\x=(x_1, \cdots, x_n)\in \mathbb{R}^n$. Then we have
\begin{multline*}\label{LargeDeviationDensity}
f(\x) = \frac{1}{(2\pi)^{n/2}} \exp\Bigg( \xi
 \sum_{j=1}^kx_{2j-1} x_{2j}- \frac{\norm{\x_{2k}}^2}{2}\left(1+O\left(\frac{1}{(\log q)^2}\right)\right) \\
  - \frac{\norm{\x^{n-2k}}^2}{2}\(1+O\( \eps n \)\) + O\pfrac{k}{(\log q)^2} \Bigg)
\end{multline*}
where $\eps$ satisfies \eqref{epsilon}.
We also have the cruder estimate
\begin{equation}\label{ApproxDensity}
f(\x)= \frac{1}{(2\pi)^{n/2}}\exp\left( - \frac{\norm{\x_{2k}}^2}{2}\left(1+O\left(\frac{1}{\log q}\)\) - \frac{\norm{\x^{n-2k}}^2}{2}\left(1+O\left( \eps n \right)\right) + O\pfrac{k}{(\log q)^2} \right).
\end{equation}
\end{pro}

To establish this result we first prove the following lemma. 
\begin{lem}\label{matrix}
Suppose $q$ is sufficiently large that $|\xi| \le 1/2$ and $\eps \le \frac{1}{4n}$. Then
 \begin{itemize}
  \item $\det C = (\det A)(1 + O(\eps n))$;
  \item $C$ is invertible, and $C^{-1}=A^{-1}+F$, where the entries of $F$ are bounded in absolute value by $8\eps$.
\end{itemize}
\end{lem}

\begin{proof}
Write $C=A+E=A(I-E')$, where $E'=-A^{-1}E$.  As $|\xi|\le 1/2$, and the inverse of $(\begin{smallmatrix} 1 & \xi \\ \xi & 1 \end{smallmatrix})$ is $(1-\xi^2)^{-1}(\begin{smallmatrix} 1 & -\xi \\ -\xi & 1 \end{smallmatrix})$,
we easily see that the entries of $E'$ are bounded in absolute value by $2\eps$.  Then $\det C = (\det A)\det (I-E')$
and, writing the determinant as a sum over permutations $\sigma\in S_n$, with $t$ the number of fixed points of $\sigma$,
\begin{eqnarray}
\det(I-E') & = & (1+ O(\eps))^n +O \Bigg( \sum_{\text{id}\ne\sigma\in S_n} (2\eps)^{n-t} (1+2\eps)^t \Bigg) \nonumber \\
& = & 1+ O(\eps n) + O \Bigg( \sum_{t=0}^{n-2} n^{n-t} (2\eps)^{n-t} \Bigg) = 1 + O(\eps n). \nonumber
\end{eqnarray}
Here we made several uses of our assumption that $|\eps n| \leq 1/4$.

Also
\[
 C^{-1} = (I + E' + (E')^2 + (E')^3 + \cdots ) A^{-1} ,
\]
where the infinite series converges because the entries of $(E')^j$ are (by an easy induction) bounded
in absolute value by $n^{j-1}(2\eps)^j \le \frac{2\eps}{2^{j-1}}$.  Hence, $C^{-1}= (I+E'')A^{-1} $, where 
the entries of $E''$ are bounded in absolute value by $4\eps$.  The second part now follows.
\end{proof}


\begin{proof}[Proof of Proposition \ref{ASYMPDF}]
First, note that $\det A= (1-\xi^2)^k$ and recall that the inverse of $(\begin{smallmatrix} 1 & \xi \\ \xi & 1 \end{smallmatrix})$ is $(1-\xi^2)^{-1}(\begin{smallmatrix} 1 & -\xi \\ -\xi & 1 \end{smallmatrix})$. Therefore, it follows from \eqref{PDFExplicit} and Lemma \ref{matrix} that
\begin{multline*}
 f(\xx) = (2\pi)^{-\frac{n}{2}} (1-\xi^2)^{-\frac{k}2}(1+O(\eps n))  \times \\
\times  \exp \Bigg\{ -\frac{1}{2(1-\xi^2)} \sum_{j=1}^k (x_{2j-1}^2+x_{2j}^2 -
 2\xi x_{2j-1} x_{2j}) - \frac12 \sum_{j=2k+1}^{n} x_j^2 +\sum_{h,j=1}^{n} \eps_{h,j} x_hx_j \Bigg\},
\end{multline*}
where $|\eps_{h,j}| \le 8\eps$ for every $1\le h,j\le n$. By the Cauchy--Schwarz inequality we obtain
$$ \Biggl|\sum_{h,j=1}^{n} \eps_{h,j} x_hx_j \Biggr| \leq 8\epsilon \Biggl(\sum_{h=1}^{n} |x_h| \Biggr)^2 \leq 8\eps n \norm{\x}^2 
=8\eps n \( \norm{\x_{2k}}^2  + \norm{\x^{n-2k}}^2 \),
 $$
and thus we deduce that
\begin{multline}\label{pdf3}
f(\xx)  =  (2\pi)^{-\frac{n}{2}} (1-\xi^2)^{-\frac{k}2}(1+O(\eps n)) \times  \\
 \; \times  \exp \Bigg\{\frac{\xi}{1-\xi^2}
 \sum_{j=1}^kx_{2j-1} x_{2j}- \norm{\x_{2k}}^2 \( \frac1{2(1-\xi^2)} + O(\eps n)\) - \norm{\x^{n-2k}}^2 \( \frac12 +O(\eps n) \) \Bigg\}. 
 \end{multline}
Invoking \eqref{XI}, we have that $\frac1{1-\xi^2}=1+O(1/\log^2 q)$.
Also, by our assumption \eqref{kn_sec4} and the estimate \eqref{epsilon},
it follows that
\[
\eps n \ll \frac{1}{\log^2 q}.
\]
Another application of Cauchy--Schwarz yields
\be\label{norm2k}
\sum_{j=1}^kx_{2j-1} x_{2j} \le \norm{\x_{2k}}^2,
\ee
and therefore
\bal
\frac{\xi}{1-\xi^2}  \sum_{j=1}^kx_{2j-1} x_{2j} &= \xi  \sum_{j=1}^kx_{2j-1} x_{2j} + O\( \xi^3 \norm{\x_{2k}}^2 \) \\ &= \xi  \sum_{j=1}^kx_{2j-1} x_{2j} + O\( (\log q)^{-3} \norm{\x_{2k}}^2 \) .
\eal
Combining these estimates with \eqref{pdf3} gives the first estimate in Proposition \ref{ASYMPDF}.

To obtain the crude estimate \eqref{ApproxDensity}, we combine \eqref{XI}
and \eqref{norm2k} to get
$$  \xi \sum_{j=1}^kx_{2j-1} x_{2j} \ll \frac{\norm{\x_{2k}}^2}{\log q}.$$
Inserting this bound into the first estimate of Proposition \ref{ASYMPDF} completes the proof.
\end{proof}

We conclude this section with an estimate for the tails of our multidimensional
Gaussian $\Z=(Z_1,\ldots,Z_n)$.

\begin{lem}\label{ztailslem}
Let $q$ be sufficiently large, and set $\mathcal{Q} := \log q$. 
Also assume \eqref{kn_sec4}. Then
\bal \p \Big(\sum_{i=2k+1}^{n} Z_{i}^2 > 10n\log(n\mathcal{Q})\Big) &\leq e^{-3n\log(n\mathcal{Q}) + O(n)}, \\ 
\p\Big(\sum_{i=1}^{2k} Z_{i}^2 > 10k\log(n\mathcal{Q})\Big)&\leq e^{-3k\log(n\mathcal{Q}) + O(k)}.
\eal
\end{lem}

Notice that both of these probabilities are negligible compared with the target probability $\exp\left(-\frac{0.6 k\log (n/k)}{\log q}\right) \frac{(n-2k)!}{n!}$ that we shall be aiming for later (see \eqref{smallbiassuffices}).

\begin{proof}[Proof of Lemma \ref{ztailslem}]
Arguing as before, the density function of $(Z_i)_{2k+1 \leq i \leq n}$ takes the form
$$ \frac{(1+O(\eps n))}{(2\pi)^{(n-2k)/2}} \exp\Big\{-\frac12 \norm{\x^{n-2k}}^2(1 + O(\eps n)) \Big\} , $$
and since $\eps n \ll \frac{1}{\log^{2}q}$ we get
\bal
\p\Big(\sum_{i=2k+1}^{n} Z_{i}^2 > 10n\log(n\mathcal{Q})\Big) & \ll  \int_{\norm{\x^{n-2k}}^2 > 10n\log(n\mathcal{Q})} \frac{e^{-\frac12 \norm{\x^{n-2k}}^2 (1 + O(\frac{1}{\log^{2}q}))}}{(2\pi)^{(n-2k)/2}} dx_{2k+1} ... dx_n \\
& \leq  e^{-3n\log(n\mathcal{Q})} \left(\int_{-\infty}^{\infty} \frac{e^{- \frac{x^2}{6} }}{\sqrt{2\pi}} dx \right)^{n-2k} = e^{-3n\log(n\mathcal{Q}) + O(n)} .
\eal

The proof of the second part of the lemma is exactly similar, using that (similarly as in Proposition \ref{ASYMPDF}) the density function of $(Z_i)_{1 \leq i \leq 2k}$ takes the form
$$ \frac{1}{(2\pi)^{k}}\exp\left\{ - \frac12 \norm{\x}^2 \left(1+O\left(\frac{1}{\log q}\right)\right) + O\left(\frac{k}{(\log q)^2}\right) \right\} . $$
\end{proof}

\section{Proof of Theorem \ref{kExtremeBias}: some reductions}\label{reductionsec}

 In this section, we carry out some preliminary reductions and calculations for the proof of Theorem \ref{kExtremeBias}.   Again, we let $(a_1, \dots, a_n)= (b_1, -b_1, \dots, b_k, -b_k, b_{k+1}, \dots, b_{n-k})$ be a tuple of residues as in section \ref{densitysec} (i.e. with the $b_i$ satisfying the conclusion of Lemma \ref{correlation-2}).

We begin by discussing the first bound \eqref{SmallBias} claimed in Theorem \ref{kExtremeBias}. We let 
$$S := \{(i,i+1) : 1 \leq i \leq 2k-1\} \cup \{(2k,j) : 2k+1 \leq j \leq n\}.$$ 
Then in view of \eqref{DensityMeasure} and the surrounding discussion in section \ref{prelimsec}, and applying Proposition \ref{normapproxpnr} with the choice $\delta = 1/(n\log q)^5$, we have
\begin{eqnarray}
\delta_{2k}(q;a_1, \dots, a_n) & = & \p\left(\frac{X(q,a_1)}{\sqrt{\var(q)}} > \frac{X(q,a_2)}{\sqrt{\var(q)}} > ... > \frac{X(q,a_{2k})}{\sqrt{\var(q)}} > \max_{2k+1 \leq j \leq n} \frac{X(q,a_j)}{\sqrt{\var(q)}} \right) \nonumber \\
& \leq & \left(1 + O\left(\frac{1}{\log q}\right)\right)\ex h^{+}_{S,\delta}(\Z) + O\left(n^2 e^{-2(n\log q)^{5/2}} + n e^{- \Theta\big(\frac{\varphi(q)^{1/3}}{n^{14} (\log q)^{10 + 2/3}}\big)}\right) , \nonumber
\end{eqnarray}
where $\Z = (Z_i)_{1 \leq i \leq n}$ are the jointly Gaussian random variables from Proposition \ref{normapproxpnr} that correspond to the residues $a_i$. If we now restrict to the range $1 \leq k \leq n/A \leq \varphi(q)^{1/50}$ in Theorem \ref{kExtremeBias}, then we have $(\frac{n^{5} \log q}{\sqrt{\varphi(q)}})^{1/3} \leq \delta \leq \frac{1}{\log^{2}n}$ in Proposition \ref{normapproxpnr}, and furthermore the ``big Oh'' term above is $\ll e^{-(n\log q)^2}$. This is negligible for \eqref{SmallBias}.

Thus, to prove \eqref{SmallBias}, it will suffice to show that
\begin{equation}\label{smallbiassuffices}
\ex h^{+}_{S,\delta}(\Z) \leq \exp\left(-\frac{0.6  k\log (n/k)}{\log q}\right) \frac{(n-2k)!}{n!} ,
\end{equation}
for $q$ sufficiently large.

We expect (since this is how we constructed things) that the smooth function $h_{S,\delta}^{+}$ should behave a lot like the indicator function $\textbf{1}_{R(S)}$. Thus, if we apply $h_{S,\delta}^{+}$ to a vector of {\em independent} standard normal random variables, and take expectations, we expect the answer to be close to $\frac{(n-2k)!}{n!}$, the probability that such a vector would satisfy the ordering dictated by $S$. Our next auxiliary lemma is an upper bound that reflects this expectation.

\begin{lem}\label{plussmoothinglem}
Let $1 \leq k \leq n/2$, and let $0 < \delta \leq \frac{1}{5 k^5 \log^{2}n}$. If $\W = (W_1, ..., W_n)$ is a vector of independent standard normal random variables, and $S$ is as above, we have
$$ \ex h_{S,\delta}^{+}(\W) \leq (1+ O(\sqrt{\delta \log(1/\delta) k^4 \log n}) )\frac{(n-2k)!}{n!} . $$
\end{lem}

We remark that our previous choice $\delta = 1/(n\log q)^5$ satisfies the hypotheses of this Lemma.

\begin{proof}[Proof of Lemma \ref{plussmoothinglem}]
By definition, we have $h_{S,\delta}^{+}(\W) \leq g(-\sqrt{\delta}) \ll e^{-1/\sqrt{\delta}}$ unless $W_i - W_j \geq -2\sqrt{\delta}$ for all pairs $(i,j) \in S$. Furthermore, in that case we still have $h_{S,\delta}^{+}(\W) \leq 1$. So recalling our construction of $S$, if we define $\widetilde{W}_i := W_i + (2k+1-i)2\sqrt{\delta}$ for $1 \leq i \leq 2k$, and define $\widetilde{W}_i := W_i$ for $2k+1 \leq i \leq n$, then
$$ \ex h^{+}_{S,\delta}(\W) \leq \p\left(\widetilde{W}_{1} > \widetilde{W}_{2} > ... > \widetilde{W}_{2k} > \max_{2k< i \leq n} \widetilde{W}_{i}\right) + O(e^{-1/\sqrt{\delta}}) . $$
Here our upper bound condition on $\delta$ implies that 
$$
e^{-1/\sqrt{\delta}}\ll \sqrt{\delta \log(1/\delta) k^4 \log n} \cdot n^{-2k} \leq \sqrt{\delta \log(1/\delta) k^4 \log n} \frac{(n-2k)!}{n!},$$
and so the ``big Oh'' term is acceptable for Lemma \ref{plussmoothinglem}.

Next, by independence the probability here is simply
\bal
&= \int_{x_1 > x_2 > ... > x_{2k}} \frac{e^{- \sum_{i=1}^{2k} (x_i - (2k+1-i)2\sqrt{\delta})^{2}/2}}{(2\pi)^k} \left( \int_{-\infty}^{x_{2k}} \frac{e^{-t^{2}/2}}{\sqrt{2\pi}} dt \right)^{n-2k} dx_1 ... dx_{2k} \\
&=  \int_{x_1 > x_2 > ... > x_{2k}} \frac{e^{- \frac12 \norm{\x_{2k}}^2 + O(k^{3/2}\sqrt{\delta} \norm{\x_{2k}} + k^{3}\delta)}}{(2\pi)^k} \left( \int_{-\infty}^{x_{2k}} \frac{e^{-t^{2}/2}}{\sqrt{2\pi}} dt \right)^{n-2k} dx_1 ... dx_{2k} . 
\eal
The part of the integral where $\norm{\x_{2k}}^2 > 10(k\log n + \log(1/\delta))$ is
$$ \ll \int_{\norm{\x_{2k}}^2 > 10(k\log n + \log(1/\delta))} \frac{e^{- \frac12 \norm{\x_{2k}}^2 +O(k^{3/2} \sqrt{\delta} \norm{\x_{2k}})}}{(2\pi)^k} dx_1 ... dx_{2k} \leq e^{-3(k\log n + \log(1/\delta)) + O(k)} , $$
which again is acceptable for Lemma \ref{plussmoothinglem}. And the complementary part of the integral, where $\norm{\x_{2k}}^2 \leq 10(k\log n + \log(1/\delta))$, is
$$ \leq e^{O\big(\sqrt{\delta k^{3} (k\log n + \log(1/\delta))}\big)} \int_{x_1 > x_2 > ... > x_{2k}} \frac{e^{- \frac12 \norm{\x_{2k}}^2}}{(2\pi)^k} \left( \int_{-\infty}^{x_{2k}} \frac{e^{-t^{2}/2}}{\sqrt{2\pi}} dt \right)^{n-2k} dx_1 ... dx_{2k} . $$
Finally, the integral here is simply $\p\left(W_1 > W_2 > ... > W_{2k} > \max_{2k< i \leq n} W_{i}\right)$, which by symmetry is equal to $\frac{(n-2k)!}{n!}$.
\end{proof}

Now we turn to the second bound \eqref{LargeBias} claimed in Theorem \ref{kExtremeBias}. Here we let 
\bal
S^{\#} &:= \{(2i-1,2i+1) : 1 \leq i \leq k-1 \} \cup \{(2i+2,2i) : 1 \leq i \leq k-1\} \\
&\qquad \cup \{(2k-1,j) : 2k \leq j \leq n\} \cup \{(i,2k) : 2k+1 \leq i \leq n\}.
\eal
With the choice $\delta = 1/(n\log q)^{5}$ as before, we see that to prove \eqref{LargeBias} it will suffice to show that
\begin{equation}\label{largebiassuffices}
\ex h^{-}_{S^{\#},\delta}(\Z) \geq \exp\left(\frac{0.6  k\log (n/k)}{\log q}\right) \frac{(n-2k)!}{n!} .
\end{equation}

\begin{lem}\label{minussmoothinglem}
Let $1 \leq k \leq n/2$, and let $0 < \delta \leq \frac{1}{5 k^5 \log^{2}n}$. If $\W = (W_1, ..., W_n)$ is a vector of independent standard normal random variables, we have
$$ (1 - O(n^{2}\sqrt{\delta}))\frac{(n-2k)!}{n!} \leq \ex h_{S^{\#},\delta}^{-}(\W) \leq \ex h_{S^{\#},\delta}^{+}(\W) \leq (1+ O(\sqrt{\delta \log(1/\delta) k^4 \log n}) )\frac{(n-2k)!}{n!} . $$
\end{lem}

\begin{proof}[Proof of Lemma \ref{minussmoothinglem}]
For the upper bound, take
\bal
\widetilde{W}_{2i-1} &= W_{2i-1} + (k+1-i) 2\sqrt{\delta} \quad (1\le i\le k), \\
\widetilde{W}_{2i} &= W_{2i} - (k+1-i) 2\sqrt{\delta} \quad (1\le i\le k), \\
\widetilde{W}_{i} &= W_{i} \quad (2k+1 \le i\le n). 
\eal
Then, if $W_i-W_j \ge -2\sqrt{\delta}$ for all $(i,j)\in S^{\#}$, we have
\[
\widetilde{W}_1 > \widetilde{W}_3 > \cdots > \widetilde{W}_{2k-1} > \max_{2k+1\le j\le n} \widetilde{W}_j, \quad
\widetilde{W}_2 < \widetilde{W}_4 < \cdots < \widetilde{W}_{2k} < \min_{2k+1\le j\le n} \widetilde{W}_j.
\]
The upper bound then follows similarly as in Lemma \ref{plussmoothinglem}. 

For the lower bound, we note that if $(x_1, \dots, x_n)\in R(S^{\#})$ is such that $|x_i-x_j|\geq 2\sqrt{\delta}$ for all $i\neq j$, then checking the definition \eqref{hg} of $h_{S^{\#},\delta}^{-}$ we obtain
$$ h_{S^{\#},\delta}^{-}(x_1, \dots, x_n) \geq \prod_{(i, j)\in S^{\#}} g(\sqrt{\delta}) \geq  1 - O(ne^{-1/\sqrt{\delta}}).$$
Hence, if we temporarily let $G$ denote the event that $|W_i - W_j| \geq 2\sqrt{\delta}$ for all $i \neq j$, then we see that $\ex h_{S^{\#},\delta}^{-}(\W)$ is
$$
\geq (1 - O(ne^{-1/\sqrt{\delta}})) \p(W_{1} > W_{3} > ... > W_{2k-1} > \{W_i\}_{2k< i \leq n} > W_{2k}>\cdots>W_4>W_2, \;  \text{and} \; G).
$$
However, by symmetry the probability here is equal to $\frac{(n-2k)!}{n!} \p(G)$, and by the union bound we have
\[
 \p(G) \geq 1 - \sum_{i \neq j} \p(|W_i - W_j| \leq 2\sqrt{\delta}) = 1 - O(n^{2}\sqrt{\delta}) . \qedhere
\]
\end{proof}

\section{Proof of Theorem \ref{kExtremeBias}: the conclusion}
In this section, we assume as in Theorem \ref{kExtremeBias} that $q$ is large, and that $1 \leq k \leq n/A \leq \varphi(q)^{1/50}$ for a sufficiently large absolute constant $A$. Let us observe at the outset that these conditions imply, in section \ref{densitysec} and specifically in Proposition \ref{ASYMPDF}, that $\eps n \ll \frac{n^2 \log^{3}q}{\varphi(q)} \leq \frac{1}{n\log^{10}q}$, say. We also recall the random variables $Z_1, ..., Z_n$ from sections \ref{densitysec} and \ref{reductionsec}, which have density function $f(x_1,  \dots, x_n)$. Finally, as in Lemma \ref{ztailslem}, we shall sometimes write $\mathcal{Q} := \log q$.

We begin by proving the statement \eqref{SmallBias} in Theorem \ref{kExtremeBias}, in the special case where $k\log(n/k) > 1000\log\log q$, say. (The case where $k\log(n/k) \leq 1000\log\log q$ is a bit more delicate because the saving factor $\exp\left(-\frac{0.6 k\log (n/k)}{\log q}\right)$ we are aiming for is only slightly smaller than 1, so that will be dealt with later.)
 In view of the discussion in section \ref{reductionsec}, to do this it will suffice to prove the upper bound \eqref{smallbiassuffices} for $\ex h^{+}_{S,\delta}(\Z)$. Let us introduce notation for the following four ``good'' events:
$$ \mathcal{A} := \Big\{ \sum_{i=2k+1}^{n} Z_i^2 \leq 10n\log(n\mathcal{Q}) \Big\} , \;\;\;\;\; \mathcal{B} := \Big\{ \sum_{i= 1}^{2k} Z_i^2 \leq 10k \log(n\mathcal{Q}) \Big\} , $$
$$ \mathcal{C} := \big\{ \max_{2k< i \leq n} Z_{i} > \sqrt{\log(n/k)} \big\} , \;\;\;\;\; \mathcal{U} := \{ \min(Z_1, ..., Z_{2k}) \geq \max_{2k< i \leq n} Z_{i} - 2\sqrt{\delta}\} , $$
where $\delta = 1/(n\log q)^5$ is the same as in section \ref{reductionsec}.

Firstly, Lemma \ref{ztailslem} and the definition of $h^{+}_{S,\delta}$ (in particular the fact that $0 \leq h^{+}_{S,\delta} \leq 1$) implies that 
\begin{align}
|\ex h^{+}_{S,\delta}(\Z) - \ex h^{+}_{S,\delta}(\Z) \textbf{1}_{\mathcal{A} \cap \mathcal{B} \cap \mathcal{U}}| \notag &\le (1-\p(\mathcal{A} \cap \mathcal{B} )) + \ex \big(
h^+_{S,\delta}(\Z) \(1-\textbf{1}_{\mathcal{U}}  \) \big) \\
&\ll e^{-3k\log(n\mathcal{Q})+O(k)}+g(-\sqrt{\delta})  \notag \\
&\ll e^{-3k\log(n\mathcal{Q})+O(k)} + e^{-(n\log q)^{5/2}} \notag \\
&\ll e^{-3k\log(n\mathcal{Q})+O(k)}\le \frac{(n-2k)!}{n!} e^{-k\log (n\mathcal{Q})+O(k)} .\label{hZABU}
\end{align}

 Next,
introduce the ``truncated'' set
\[
S' = \{(i,i+1) : 1\le i\le 2k-1 \},
\]
considered as a subset of $\{1,\ldots,2k\}^2$, the associated set
$R(S') \subset \RR^{2k}$ and function $h_{S',\delta}^+: \RR^{2k}\to \RR$.
Using the crude estimate from Proposition \ref{ASYMPDF} we have
\begin{eqnarray}\label{notccalc}
\ex h^{+}_{S,\delta}(Z) \textbf{1}_{\mathcal{A} \cap \mathcal{B} \cap \mathcal{U} \cap \{\mathcal{C} \; \text{fails}\}}  & \leq &  e^{O(\frac{k}{\log^{2}q})} \idotsint\limits_{\norm{\x_{2k}}^2 \leq 10k\log(n\mathcal{Q})} h_{S',\delta}^{+}(x_1, \dots, x_{2k}) \frac{e^{-\frac{\norm{\x_{2k}}^2}{2}(1 + O(\frac{1}{\log q}))}}{(2\pi)^k} dx_1\cdots dx_{2k}  \nonumber \\
&& \times \idotsint\limits_{\substack{\max_{2k+1 \leq i \leq n} x_i \leq \sqrt{\log(n/k)}, \\ \norm{\x^{n-2k}}^2 \leq 10n\log(n\mathcal{Q})}} \frac{e^{- \frac{\norm{\x^{n-2k}}^2}{2} (1 + O(\frac{1}{n\log^{10}q}))}}{(2\pi)^{(n-2k)/2}} dx_{2k+1} ... dx_n.
\end{eqnarray}
 Now the integral on the first line here is
$$ \leq e^{O(\frac{k\log(n\mathcal{Q})}{\log q})} \int h_{S',\delta}^{+}(x_1, \dots, x_{2k}) \frac{e^{-\frac{\norm{\x_{2k}}^2}{2} }}{(2\pi)^k} dx_1\cdots dx_{2k} \ll e^{O(\frac{k\log(n\mathcal{Q})}{\log q})} \frac{1}{(2k)!} = \frac{e^{O(k)}}{k^{2k}} , $$
where the second inequality follows from Lemma \ref{plussmoothinglem} (with $n$ replaced by $2k$), and the final equality follows from Stirling's formula. Meanwhile, the integral on the second line is
$$ \ll \left( \int_{-\infty}^{\sqrt{\log(n/k)}} \frac{e^{- \frac{x^2}{2}}}{\sqrt{2\pi}} dx \right)^{n-2k} = \left( 1 - \Theta\left(\frac{e^{- \log(n/k)/2}}{\sqrt{\log(n/k)}} \right) \right)^{n-2k} = \exp\Bigg\{- \Theta\( \sqrt{\frac{nk}{\log(n/k)}} \) \Bigg\} . $$
(Notice here that $n/k \geq A$ is large, under our hypotheses.) So putting things together, we have shown that
\begin{eqnarray}\label{cfailsdisplay}
\ex \left\{h^{+}_{S,\delta}(\Z) \textbf{1}_{\mathcal{A} \cap \mathcal{B} \cap \mathcal{U} \cap \{\mathcal{C} \; \text{fails}\}}\right\} \leq \frac{e^{O(k)}}{k^{2k}} e^{- \Theta\big( \sqrt{\frac{nk}{\log(n/k)}}\big) } &=& \frac{1}{n^{2k}} e^{O(k\log(n/k)) - \Theta\big(k \sqrt{\frac{n/k}{\log(n/k)}}\big)} \nonumber \\
&\leq& \frac{1}{n^{2k}} e^{-1000k\log(n/k)} \nonumber \\
&\leq& \frac{(n-2k)!}{n!}e^{-1000k\log(n/k)}  ,
\end{eqnarray}
say.

On the other hand, the first part of Proposition \ref{ASYMPDF} implies that on those tuples $(x_1, ..., x_n)$ corresponding to all four events $\mathcal{A},\mathcal{B},\mathcal{C}, \mathcal{U}$, the density function $f(x_1,  \dots, x_n)$ satisfies
\bal
f(\x) &\leq \frac{e^{\xi k\log(n/k) + O(\frac{k}{\log^{2}q})}}{(2\pi)^{n/2}} \exp\left\{ - \frac{\norm{\x_{2k}}^2}{2}\Big(1+O\Big(\frac{1}{\log^2 q}\Big)\Big) - \frac{\norm{\x^{n-2k}}^2}{2}\Big(1+O\Big( \frac{1}{n\log^{10}q} \Big)\Big) \right\}  \\
&= \frac{e^{\xi k\log(n/k) + O\big(\frac{k\log(n\mathcal{Q})}{\log^{2}q}\big)}}{(2\pi)^{n/2}} \exp\left\{ - \frac{\norm{\x}^2}{2} \right\} , 
\eal
where $\xi \sim -\frac{\log 2}{\log q}$ (recall \eqref{XI}). The crucial thing to notice here is the emergence of the bias term $e^{\xi k\log(n/k)}$, which came from the non-trivial correlations of size $\xi$ amongst pairs $(Z_{2i-1},Z_{2i})_{1 \leq i \leq k}$, together with the fact that we have arranged to have $Z_1, ..., Z_{2k} \geq \sqrt{\log(n/k)} - 2\sqrt{\delta}$. Using this upper bound, we get
\begin{eqnarray}
\ex\left\{h^{+}_{S,\delta}(\Z) \textbf{1}_{\mathcal{A} \cap \mathcal{B} \cap \mathcal{C} \cap \mathcal{U}} \right\}& \leq & e^{\xi k\log(n/k) + O(\frac{k\log(n \mathcal{Q})}{\log^{2}q})} \int h_{S,\delta}^{+}(x_1, \dots, x_{n}) \frac{e^{- \frac{\norm{\x}^2}{2}}}{(2\pi)^{n/2}} dx_1 ... dx_n \nonumber \\
&=& e^{\xi k\log(n/k) + O(\frac{k\log(n\mathcal{Q})}{\log^{2}q})} \ex h_{S,\delta}^+ (\W) 
 \leq  e^{\xi k\log(n/k) + O(\frac{k\log(n\mathcal{Q})}{\log^{2}q})} \frac{(n-2k)!}{n!}, \nonumber
\end{eqnarray}
where $\W=(W_1,\ldots,W_n)$ is a vector of independent standard Gaussian
random variables, and
where the second inequality follows from Lemma \ref{plussmoothinglem} on noting that 
 $$\(1+ O\(\sqrt{\delta \log(1/\delta) k^4 \log n}\) \) \leq \(1+ O\(\sqrt{\frac{\log^2 (n\mathcal{Q})}{n\log^{5}q}}\) \) \leq e^{O(\frac{k\log(n\mathcal{Q})}{\log^{2}q})}.$$
The error in the exponent satisfies
\[
O\pfrac{k\log(n\mathcal{Q})}{\log^{2}q} = O\( k\(\frac{\log\log q}{\log^2 q}+\frac{1}{\log q}\)\) = O\pfrac{k}{\log q},
\] 
and thus
\be\label{ABCU}
\ex\left\{h^{+}_{S,\delta}(\Z) \textbf{1}_{\mathcal{A} \cap \mathcal{B} \cap \mathcal{C} \cap \mathcal{U}} \right\} \le e^{\xi k\log(n/k) +O\pfrac{k}{\log q}}
\frac{(n-2k)!}{n!}.
\ee
Again recall that $\xi \sim -\frac{\log 2}{\log q}$ as $q\to\infty$.  Hence,
for $q$ sufficiently large, and $A$ sufficiently large,
 the $O$-term in the exponent is 
smaller than $0.01 |\xi k\log (n/k)|$ and consequently
\[
\xi k \log(n/k) +O\pfrac{k}{\log q} \le - 0.65 \frac{k \log(n/k)}{\log q}.
\]
Adding this together with \eqref{hZABU} and \eqref{cfailsdisplay}, and recalling that we are in the special case where $k\log(n/k) > 1000\log\log q$,
we indeed  have a bound $\leq \exp\left(-\frac{0.6 k\log (n/k)}{\log q}\right) \frac{(n-2k)!}{n!}$ as we wanted for \eqref{smallbiassuffices}.

\vspace{12pt}
Next we turn to the other case where $k\log(n/k) \leq 1000\log\log q$. Notice that in this case, in the definition of the event $\mathcal{B}$ we have $10k\log(n\mathcal{Q}) = 10k(\log k + \log(n/k) + \log\mathcal{Q}) \ll (\log\log q)^2$, and so we get $||\x_{2k}||^2 \ll (\log\log q)^2$. Thus, if $W_1, ..., W_n$ are independent standard normal random variables, and if we write `primed' events (e.g., $\mathcal{A}'$) for the corresponding events with each $Z_i$ replaced by $W_i$, then the above calculations in fact imply that
\begin{eqnarray}
\ex h^{+}_{\delta}(\Z) \textbf{1}_{\mathcal{A} \cap \mathcal{B} \cap \mathcal{U} \cap \mathcal{C}} & \leq & e^{\xi k\log(n/k) + O(\frac{k}{\log q})} \ex h^{+}_{\delta}(\W) \textbf{1}_{\mathcal{A}' \cap \mathcal{B}' \cap \mathcal{U}' \cap \mathcal{C}'} \nonumber \\
& = & (1 + \xi k\log(n/k) + O(\frac{k}{\log q})) \ex h^{+}_{\delta}(\W) \textbf{1}_{\mathcal{A}' \cap \mathcal{B}' \cap \mathcal{U}' \cap \mathcal{C}'} , \nonumber
\end{eqnarray}
and that (using the crude density estimate from Proposition \ref{ASYMPDF}, as in \eqref{notccalc})
\begin{eqnarray}
&& |\ex h^{+}_{\delta}(\Z) \textbf{1}_{\mathcal{A} \cap \mathcal{B} \cap \mathcal{U} \cap \{\mathcal{C} \; \text{fails}\}} - \ex h^{+}_{\delta}(\W) \textbf{1}_{\mathcal{A}' \cap \mathcal{B}' \cap \mathcal{U}' \cap \{\mathcal{C}' \; \text{fails}\}}| \nonumber \\
& \ll & \int_{\mathcal{A} \cap \mathcal{B} \cap \mathcal{U} \cap \{\mathcal{C} \; \text{fails}\}} h_{\delta}^{+}(x_1, \dots, x_{n}) \frac{e^{- \frac{||\x||^2}{2}}}{(2\pi)^{n/2}} \left( \frac{k}{\log^{2}q} + \frac{||\x_{2k}||^2}{\log q} \right) dx_1 ... dx_{n} . \nonumber
\end{eqnarray}
The part of this integral where $||\x_{2k}||^2 \leq 10k\log k$ contributes $\ll \frac{k\log k}{\log q} \ex h^{+}_{\delta}(\W) \textbf{1}_{\mathcal{A}' \cap \mathcal{B}' \cap \mathcal{U}' \cap \{\mathcal{C}' \; \text{fails}\}}$, which is $\ll \frac{1}{\log q} \frac{(n-2k)!}{n!}$ by the calculations leading to \eqref{cfailsdisplay}. Meanwhile, another short calculation shows the other part of the integral contributes
\begin{eqnarray}
& \ll & \frac{1}{\log q} \int_{||\x_{2k}||^2 > 10k\log k} \frac{e^{- \frac{||\x_{2k}||^2}{2}}}{(2\pi)^{k}} ||\x_{2k}||^2 dx_1 ... dx_{2k}  \cdot \p(\max_{2k< i \leq n} W_{i} \leq \sqrt{\log(n/k)}) \nonumber \\
& \ll & \frac{1}{\log q} e^{-3k\log k + O(k)} e^{- \Theta\big( \sqrt{\frac{nk}{\log(n/k)}}\big) } \ll \frac{1}{\log q} \frac{(n-2k)!}{n!} . \nonumber
\end{eqnarray}
The calculations leading to \eqref{cfailsdisplay} further imply that $\xi k\log(n/k) \ex h^{+}_{\delta}(\W) \textbf{1}_{\mathcal{A}' \cap \mathcal{B}' \cap \mathcal{U}' \cap \{\mathcal{C}' \; \text{fails}\}} \ll \frac{k \log(n/k)}{\log q}\ex h^{+}_{\delta}(\W) \textbf{1}_{\mathcal{A}' \cap \mathcal{B}' \cap \mathcal{U}' \cap \{\mathcal{C}' \; \text{fails}\}} \ll \frac{1}{\log q} \frac{(n-2k)!}{n!}$. So we can collect together the above computations, along with \eqref{hZABU}, in the form
\begin{eqnarray}
\ex h^{+}_{\delta}(\Z) & = & \ex h^{+}_{\delta}(\Z) \textbf{1}_{\mathcal{A} \cap \mathcal{B} \cap \mathcal{U} \cap \mathcal{C}} + \ex h^{+}_{\delta}(\Z) \textbf{1}_{\mathcal{A} \cap \mathcal{B} \cap \mathcal{U} \cap \{\mathcal{C} \; \text{fails}\}} + (\ex h^{+}_{\delta}(\Z) - \ex h^{+}_{\delta}(\Z) \textbf{1}_{\mathcal{A} \cap \mathcal{B} \cap \mathcal{U}}) \nonumber \\
& \leq & (1 + \xi k\log(n/k) + O(\frac{k}{\log q})) \ex h^{+}_{\delta}(\W) \textbf{1}_{\mathcal{A}' \cap \mathcal{B}' \cap \mathcal{U}'} + O(\frac{1}{\log q} \frac{(n-2k)!}{n!}) . \nonumber
\end{eqnarray}
And Lemma \ref{plussmoothinglem} implies that $\ex h^{+}_{\delta}(\W) \textbf{1}_{\mathcal{A}' \cap \mathcal{B}' \cap \mathcal{U}'} \leq \ex h^{+}_{\delta}(\W) \leq (1+ O(\sqrt{\frac{\log\log q}{\log^{4}q}}) ) \frac{(n-2k)!}{n!}$, which suffices to establish \eqref{smallbiassuffices} in the case where $k\log(n/k) \leq 1000\log\log q$.

\vspace{12pt}
Next we shall prove the other statement \eqref{LargeBias} in Theorem \ref{kExtremeBias}. To do this it will suffice to prove the lower bound \eqref{largebiassuffices} for $\ex h_{S^{\#},\delta}^{-}(\Z)$. In addition to the three typical events $\mathcal{A}, \mathcal{B}, \mathcal{C}$ considered above, we introduce the following:
$$ \mathcal{D} := \left\{ \min_{2k< i \leq n} Z_{i} < - \sqrt{\log(n/k)} \right\} , $$
$$ \mathcal{V} := \big\{Z_1, Z_3, ..., Z_{2k-1} \geq \max_{2k< i \leq n} Z_{i} - 2\sqrt{\delta}, \;\;\; \text{and} \;\;\; Z_2, Z_4, ..., Z_{2k} \leq \min_{2k< i \leq n} Z_{i} + 2\sqrt{\delta}\big\} . $$

Similarly as above, on those tuples $(x_1, ..., x_n)$ corresponding to all the events $\mathcal{A}, \mathcal{B}, \mathcal{C}, \mathcal{D}, \mathcal{V}$, the density function $f(x_1,  \dots, x_n)$ satisfies
$$ f(\x) \geq \frac{e^{-\xi k\log(n/k) + O\big(\frac{k\log(n\mathcal{Q})}{\log^{2}q}\big)}}{(2\pi)^{n/2}} \exp\left( - \frac{\norm{\x}^2}{2} \right) . $$
 Again, the key thing to notice here is the large bias factor $e^{-\xi k\log(n/k)}$, which emerged because we arranged to have 
\bal
\min(Z_1, Z_3, ..., Z_{2k-1}) &\geq \sqrt{\log(n/k)} - 2\sqrt{\delta},\; \text{ and} \\
\max(Z_2, Z_4, ..., Z_{2k}) &\leq -(\sqrt{\log(n/k)} - 2\sqrt{\delta}).
\eal
 So if we again let $\W = (W_1, ..., W_n)$ denote a vector of independent standard Gaussian random variables, we have that
\begin{eqnarray}
\ex h^{-}_{S^{\#},\delta}(\Z) & \geq & \ex h^{-}_{S^{\#},\delta}(\Z) \textbf{1}_{\mathcal{A} \cap \mathcal{B} \cap \mathcal{V}} = \ex h^{-}_{S^{\#},\delta}(\Z) \textbf{1}_{\mathcal{A} \cap \mathcal{B} \cap \mathcal{C} \cap \mathcal{D} \cap \mathcal{V}} + \ex h^{-}_{S^{\#},\delta}(\Z) \textbf{1}_{\mathcal{A} \cap \mathcal{B} \cap \mathcal{V} \cap \{\mathcal{C} \; \text{or} \; \mathcal{D} \; \text{fails}\}} \nonumber \\
& \geq & e^{-\xi k\log(n/k) + O\big(\frac{k\log(n\mathcal{Q})}{\log^{2}q}\big)}
  \ex h^{-}_{S^{\#},\delta}(\W) \textbf{1}_{\mathcal{A}' \cap \mathcal{B}' \cap \mathcal{C}' \cap \mathcal{D}' \cap \mathcal{V}'} + \ex h^{-}_{S^{\#},\delta}(\Z) \textbf{1}_{\mathcal{A} \cap \mathcal{B} \cap \mathcal{V} \cap \{\mathcal{C} \; \text{or} \; \mathcal{D} \; \text{fails}\}} , \nonumber
\end{eqnarray}
where the `primed' events (e.g., $\mathcal{A}'$) are the corresponding events with each $Z_i$ replaced by $W_i$.
As before, for $A$ and $q$ large enough, the $O$-term in the exponent is negligible and we have $e^{-\xi k\log(n/k) + O\big(\frac{k\log(n\mathcal{Q})}{\log^{2}q}\big)} \geq e^{0.65 k \frac{\log(n/k)}{\log q}}$, and so
\begin{eqnarray}\label{bigbiasorg}
\ex h^{-}_{S^{\#},\delta}(\Z) & \geq & e^{0.65 k \frac{\log(n/k)}{\log q}} \ex h^{-}_{S^{\#},\delta}(\W) \textbf{1}_{\mathcal{A}' \cap \mathcal{B}' \cap \mathcal{V}'} \nonumber \\
&& + \ex h^{-}_{S^{\#},\delta}(\Z) \textbf{1}_{\mathcal{A} \cap \mathcal{B} \cap \{\mathcal{C} \; \text{or} \; \mathcal{D} \; \text{fails}\} \cap \mathcal{V}} - e^{0.65 k \frac{\log(n/k)}{\log q}} \ex h^{-}_{S^{\#},\delta}(\W) \textbf{1}_{\mathcal{A}' \cap \mathcal{B}' \cap \{\mathcal{C}' \; \text{or} \; \mathcal{D}' \; \text{fails}\} \cap \mathcal{V}'} .
\end{eqnarray}
Furthermore (the proofs of) Lemmas \ref{ztailslem} and \ref{minussmoothinglem} together imply, with our choice $\delta = 1/(n\log q)^5$, that
$$  \ex h^{-}_{S^{\#},\delta}(\W) \textbf{1}_{\mathcal{A}' \cap \mathcal{B}' \cap \mathcal{V}'} \geq (1 - O(n^2\sqrt{\delta})) \frac{(n-2k)!}{n!} + O(e^{-3k\log(n\mathcal{Q}) + O(k)}) = (1 - O(\frac{1}{\log^{5/2}q})) \frac{(n-2k)!}{n!} . $$

We can also mimic the calculations leading to \eqref{cfailsdisplay}, this time using Lemma \ref{minussmoothinglem} in place of Lemma \ref{plussmoothinglem} to bound $\int h_{S^{\#},\delta}^{-}(x_1, \dots, x_{2k}) \frac{e^{-\frac12 \norm{\x_{2k}}^2 }}{(2\pi)^k} dx_1\cdots dx_{2k}$, and obtain that
$$ \ex \left\{h^{-}_{S^{\#},\delta}(\W) \textbf{1}_{\mathcal{A}' \cap \mathcal{B}' \cap \mathcal{V}' \cap \{\mathcal{C}' \; \text{or} \; \mathcal{D}' \; \text{fails}\}}\right\} \leq \frac{1}{n^{2k}} e^{-1000k\log(n/k)} \le \frac{(n-2k)!}{n!}e^{-1000k\log(n/k)}  , $$
and the same for $\ex \left\{h^{-}_{S^{\#},\delta}(\Z) \textbf{1}_{\mathcal{A} \cap \mathcal{B} \cap \mathcal{V} \cap \{\mathcal{C} \; \text{or} \; \mathcal{D} \; \text{fails}\}}\right\}$. Hence when $k\log(n/k) > 1000\log\log q$, the final two terms in \eqref{bigbiasorg} make a negligible contribution, and the desired lower bound \eqref{largebiassuffices} for $\ex h^{-}_{S^{\#},\delta}(\Z)$ follows. In the remaining case where $k\log(n/k) \leq 1000\log\log q$, we can again duplicate our previous approach and show the {\em difference} of the final two terms in \eqref{bigbiasorg} is $\ll \frac{1}{\log q} \frac{(n-2k)!}{n!}$, which is negligible for \eqref{largebiassuffices}.
\qed

\end{document}